\documentclass[11pt,a4paper,reqno]{amsart}
\usepackage{amsthm,amsmath,amsfonts,amssymb,amsxtra,appendix,bookmark,dsfont,latexsym,bm,hyperref,delarray,color,euscript,amsgen,amsbsy,amsopn,amscd,latexsym,mathrsfs}
\usepackage{enumitem}

\makeatletter

\allowdisplaybreaks

\newtheorem{theorem}{Theorem}[section]
\newtheorem{lemma}[theorem]{Lemma}
\newtheorem{proposition}[theorem]{Proposition}

\theoremstyle{definition}
\newtheorem{definition}{Definition}

\theoremstyle{remark}
\newtheorem{remark}[theorem]{Remark}

\newcommand\R{{\ensuremath {\mathbb R} }}
\newcommand\C{{\ensuremath {\mathbb C} }}
\newcommand\N{{\ensuremath {\mathbb N} }}

\newcommand\1{{\ensuremath {\mathds 1} }}

\newcommand\nn{\nonumber}

\newcommand{\bS}{\mathbb{S}}

\newcommand{\wto}{\rightharpoonup}
\newcommand{\cS}{\mathcal{S}}

\newcommand{\cQ}{\mathcal{Q}}

\newcommand{\sC}{\mathscr{C}}

\newcommand{\cA}{\mathcal{A}}

\newcommand{\cW}{\mathcal{W}}

\newcommand{\cD}{\mathcal{D}}

\newcommand{\cL}{\mathcal{L}}

\newcommand{\cG}{\mathcal{G}}

\numberwithin{equation}{section}

\begin{document}

\title[NLS equations with critical Hardy potential and Choquard nonlinearity]{Nonlinear Schr\"odinger equations with critical Hardy potential and Choquard nonlinearity}

\author[P.-T. Nguyen]{Phuoc-Tai Nguyen}
\address{Department of Mathematics and Statistics, Masaryk University, Kotl\'a\v rsk\'a 2, 61137 Brno, Czech Republic}
\email{ptnguyen@math.muni.cz}

\author[T. D. Tran]{Tuan Dat Tran}
\address{Department of Mathematics and Statistics, Masaryk University, Kotl\'a\v rsk\'a 2, 61137 Brno, Czech Republic}
\email{554671@mail.muni.cz}

\begin{abstract} We study the Cauchy problem for the nonlinear Schr\"odinger equation characterized by contrasting effects between the concentration at the origin of a critical Hardy potential and the intrinsic nonlocality of a Choquard nonlinearity. We prove the existence of a ground state solution through optimizers of an interpolation Hardy-Gagliardo-Nirenberg inequality and derive a non-existence result via Poho\v zaev identities. Using these results, we provide various criteria for the global existence and finite-time blow-up for the problem in the energy-subcritical regime. Finally, we establish a key compactness result, which enables us to obtain a characterization of finite-time blow-up solutions with minimal mass.

\medskip
	
 \medskip

\end{abstract}
\maketitle

\tableofcontents
\section{Introduction}

In this paper, we study the Cauchy problem for the nonlinear Schr\"odinger (NLS) equation with an inverse-square potential and a focusing Choquard nonlinearity of the form
\begin{equation} \label{eq:NLS} 
	\left\{  \begin{aligned}
		i \partial_t u &= \cL_{\mu_0} u - (I_\alpha * |u|^p)|u|^{p-2}u, \quad t>0,\; x\in \R^d, \\
		u(0,x) &=u_0(x), \quad x\in \R^d.
	\end{aligned} \right. 
\end{equation}
Here 
$$ d \geq 3, \quad \mu_0 = \frac{(d-2)^2}{4}, \quad \cL_{\mu_0} = -\Delta - \frac{\mu_0}{|x|^2}, \quad 0<\alpha<d, \quad \frac{d+\alpha}{d} < p < \frac{d+\alpha}{d-2},
$$
 $u: \R_+ \times \R^d \to \C$, $u_0: \R^d \to \C$,  and $I_\alpha$ denotes the Riesz potential defined by
$$ I_\alpha(x) = \frac{{\Gamma \left( {\frac{{d - \alpha }}{2}} \right)}}{{\Gamma \left( {\frac{\alpha }{2}} \right){\pi ^{\frac{d}{2}}}{2^\alpha }}}|x|^{-(d-\alpha)}, \quad x \neq 0,
$$
with $\Gamma$ being the Gamma function.

Nonlinear Schr\"odinger (NLS) equations without potential (i.e. $\mu=0$) and with a Choquard nonlinearity have been extensively studied due to their interest in many fields. When $p=2$, the equation in \eqref{eq:NLS} reduces to the well-known Schr\"odinger-Hartree equation.  The typical model in the physical case ($d=3$ and $p=\alpha=2$) was first proposed by Pekar in the context of quantum mechanics (see \cite{LewRou-13} for references). Such equations also play an important role in the study of Hartree-Fock theory (see \cite{Lio-80} and related works \cite{BGGM-03,EESY-04}). The wellposedness and dynamics of solutions to NLS equations with Choquard nonlinearities have been well understood; see the papers \cite{MXZ-07,FenYua-15} (and references therein) and the excellent textbook \cite{Caz-03}.

In the last decade, there has been a fast growing interest in NLS equations in $\R^d$ ($d \geq 3$) governed by a Schr\"odinger operator involving an inverse-square potential of the form
$$ \cL_\mu := -\Delta - \mu |x|^{-2}, \quad \mu \in \R.
$$  
The singular term $\mu |x|^{-2}$ represents a borderline potential and cannot be treated as a lower-order perturbation of $-\Delta$, thereby significantly affecting the solvability and dynamics of the associated Cauchy problem. Moreover, the potential $\mu |x|^{-2}$ is also connected to the Hardy inequality
\begin{equation} \label{hardy-ineq} \cL_{\mu_0} \geq 0 \quad \text{in } L^2(\R^d) \quad \text{with }  \mu_0 = \frac{(d-2)^2}{4},
\end{equation}
hence it is referred to as \textit{subcritical (resp. critical) Hardy potential} if $\mu<\mu_0$ (resp. $\mu=\mu_0$). Thanks to the Hardy inequality \eqref{hardy-ineq}, by the Friedrichs method, for any $\mu \leq \mu_0$,  $\cL_{\mu}$ can be extended to non-negative self-adjoint operator on $L^2(\R^d)$. The quadratic form domain $\cQ_{\mu}$ of $\cL_{\mu}$ is a Hilbert space with the norm
\begin{equation} \label{def:Q-norm}
	\|u\|_{\cQ_{\mu}} :=\left( \|\sqrt{\cL_{\mu}}u\|_{L^2}^2  + \|u\|^2_{L^2}  \right)^{\frac{1}{2}}, 
\end{equation}
where
\begin{equation} \label{sqrtL} \|\sqrt{\cL_{\mu}}u\|_{L^2}^2 = \int_{\R^d} \left(|\nabla u|^2 - \frac{\mu}{|x|^2}|u|^2\right)dx.  
\end{equation}

 
The Cauchy problem for the NLS equation with a focusing Choquard type nonlinearity is of the form
\begin{equation} \label{eq:NLS-mu} 
	\left\{  \begin{aligned}
		i \partial_t u &= \cL_{\mu} u - (I_\alpha * |u|^p)|u|^{p-2}u, \quad t>0,\; x\in \R^d, \\
		u(0,x) &=u_0(x), \quad x\in \R^d.
\end{aligned} \right. \end{equation}
	
A distinctive feature of problem \eqref{eq:NLS-mu} lies in the competition between the Hardy potential and the Choquard nonlinearity. More specifically, the Hardy potential exhibits a strong concentration effect at the origin, which calls for  localization techniques to handle the singularity separably. In contrast, the Choquard nonlinearity is intrinsically nonlocal and heavily depends on the behavior of solutions over the whole space. The intricate interplay between these terms leads to substantial analytical difficulties and significantly complicates or invalidates standard approaches.

The space $\cQ_{\mu}$ is a natural energy space associated to problem \eqref{eq:NLS-mu}. Due to the presence of the Hardy potential, problem \eqref{eq:NLS-mu} and the energy space $\cQ_{\mu}$ are not space-translation invariant.

When $\mu<\mu_0$,  the Hardy inequality \eqref{hardy-ineq} implies that the term $\int_{\R^d}\frac{\mu}{|x|^2}|u|^2dx$ in \eqref{sqrtL} can be controlled by $\int_{\R^d}|\nabla u|^2dx$, hence the norm defined in \eqref{def:Q-norm} is equivalent to the $H^1$-norm. This allows for an adaptation of standard techniques used in $H^1$ setting with an additional treatment of the Hardy potential. Numerous works have been devoted to investigating NLS \eqref{eq:NLS-mu} with subcritical Hardy potential. Among many others, these include \cite{Suz-13,Li-20} for global existence and finite-time blow-up dichotomy, \cite{CheLuLu-19} regarding the blow-up phenomena in the mass critical setting, \cite{GuoTan} for the existence and qualitative properties of solutions to a variant of ground state equations, \cite{Suz-13} for local and global existence via energy method, \cite{Suz-14} for finite-time blow-up criteria.

The case of critical Hardy potential is much more challenging due to the fact that $\cQ_{\mu_0}$ is strictly larger than $H^1(\R^d)$. In fact, the following continuous embeddings (see \cite[Theorem 1.2]{Fra-09} and \cite[Page 16]{Suz-16}) hold
\begin{equation} \label{eq:Q-in-Lp} H^1(\R^d) \varsubsetneq \cQ_{\mu_0}  \varsubsetneq H^s(\R^d) \cap L^q(\R^d), \quad \forall s \in (0,1), \, \forall q \in [2,2^*), \quad \text{with } 2^* = \frac{2d}{d-2}.  
\end{equation} 
In this case, the kinetic and potential parts on the right-hand side of \eqref{sqrtL} should always be tied together and treated as a whole to allow for possible implicit cancellations and to ensure the finiteness of the norm \eqref{sqrtL}. These factors together invalidate classical tools, requiring new or refined techniques to investigate  problem \eqref{eq:NLS-mu}. The dynamics and related results of problem \eqref{eq:NLS} have been understood in the case $p=2$; see \cite{Miz-17}) for local and global solvability, \cite[Remark 4.2]{Suz-16} for Virial identities and \cite[Theorem 4.1 (B3)]{Suz-17} for blow-up criteria. Nevertheless, to the best of our knowledge, the case $p>2$ remains largely unexplored, which is our motivation to study problem \eqref{eq:NLS}in this case. To this end, we will refine various techniques and tools, including optimizers for the Hardy-Galiardo-Nirenberg inequallity, Pohoz\v{a}ev identities for ground states, a compactness result, and adapt the methods in \cite{Wei-83,HolRou-07,KenMer-08,Suz-17,Mer-93,MukNamNgu-2021} to the context of critical Hardy potential. The detailed statement of the main results of the paper will be presented in the next subsection.

\subsection{Main results}
Our first main theorem concerns the ground state equation.

\begin{theorem}  \label{th:groundstate} Assume $0<\alpha<d$ and $p>1$. Consider the ground state equation 
	\begin{equation} \label{eq:EL-u}
		\mathcal{L}_{\mu_0} w + w =  (I_\alpha * |w|^p) |w|^{p-2}w \quad \text{in } \R^d.
	\end{equation}	
	
	\noindent  {\sc 1. Existence and Pohoz\v aev identities.} Assume $\frac{d+\alpha}{d}<p<\frac{d+\alpha}{d-2}$. Then equation \eqref{eq:EL-u} admits a positive radial solution in $\cQ_{\mu_0}$ which is monotonically decreasing in $|x|$. Moreover, any solution $w \in \cQ_{\mu_0}$ to equation \eqref{eq:EL-u} satisfies the Pohoz\v{a}ev identities  
	\begin{align} \label{eq:Pohozaev-1}
		& \| (I_\alpha*|w|^p)	|w|^p \|_{L^1} =  \frac{2p}{d+\alpha- (d-2)p}\|w\|_{L^2}^2, \\ \label{eq:Pohozaev-2}
		&\| \sqrt{\cL_{\mu_0}}w\|_{L^2} = \left(\frac{dp-d-\alpha}{d+\alpha-(d-2)p}\right)^{\frac{1}{2}} \| w\|_{L^2}.
	\end{align}
	
	\noindent {\sc	2. Nonexistence.} If $1<p \leq \frac{d+\alpha}{d}$ or $p \geq \frac{d+\alpha}{d-2}$, then there is no nontrivial solution in $\cQ_{\mu_0} \cap  L^{\frac{2dp}{d+\alpha}}(\R^d) \cap W_{\mathrm{loc}}^{2,q}(\R^d)$ $($with some $q>1)$ to equation \eqref{eq:EL-u}. \smallskip
	
	\noindent {\sc	3. Upper estimates and asymptotic behaviors.}  Assume  $(d-4)_+<\alpha<d$ and  $2<p<\frac{d+\alpha}{d-2}$. If  $w \in \cQ_{\mu_0}$ is a positive solution to equation \eqref{eq:EL-u} then $w \in C^2(\R^d \setminus \{0\})$ and 
	\begin{equation} \label{est:upbound-w}
		w(x) \leq C|x|^{-(d-2)}\1_{B_1}(x) + C|x|^{-\frac{d-1}{2}}e^{-\frac{|x|}{2}}\1_{B_1^c}(x), \quad \forall x \neq 0.	
	\end{equation}
	Moreover, if $w \in \cQ_{\mu_0}$ is a positive radial solution to equation \eqref{eq:EL-u} then
	\begin{equation} \label{asymp-behav} \lim_{|x| \to 0}|x|^{\frac{d-2}{2}}w(x) \in (0,\infty) \quad \text{and} \quad \lim_{|x| \to \infty}|x|^{\frac{d-1}{2}}e^{|x|}w(x) \in (0,\infty).
	\end{equation}
\end{theorem}	

We remark that the existence part in statement 1 is derived from the existence of an optimizer to the Hardy-Gagliardo-Nirenberg inequality  \eqref{eq:CHGN-def} and a suitable scaling (see Lemma \ref{lem:HGN}). 

The Pohoz\v aev identities \eqref{eq:Pohozaev-1}--\eqref{eq:Pohozaev-2} and statement 2 follows from Proposition \ref{prop:Pohozaev}. Due to the fact that the critical Hardy potential cannot be handled independently and separably from the Laplacian, the usual argument, namely by multiplying \eqref{eq:NLS} by the test function $\nabla w \cdot x$, is inapplicable. To overcome this difficulty, we employ the ground state representation $v(x)=|x|^{\frac{d-2}{2}}w(x)$ and convert equation \eqref{eq:EL-u} to a weighted equation of divergent form satisfied by $v$ (see \eqref{eq:transform-v}). Moreover, we also localize the singularity of the Hardy potential by using cut-off-function techniques. Then by carrying out carefully computations for the weighted divergent operator and the Choquard nonlinearity, we obtain the desired Pohoz\v aev identities. 

In order to establish the upper bounds \eqref{est:upbound-w}, we employ the Moser iteration for equation \eqref{eq:transform-v} to derive that $v \in L^\infty(\R^d)$, which implies that $w(x) \lesssim |x|^{-\frac{d-2}{2}}$ for $x \neq 0$. On one hand, this is a sharp estimate for $w$ near the origin. On the other hand, this allows to control the Choquard nonlinearity near infinity, which, combined with the argument from \cite{MorVan-2013}, yields the decay estimate for $w$ in \eqref{est:upbound-w}. Note that the condition $p>2$ is needed in both Moser iteration and the argument leading to the decay estimate. 

If we assume additionally that $w$ is radial then so is $v$. Performing a logarithmic transformation $\tilde v(\rho) = v(r)$ with $\rho=(-\ln(r))^{-\frac{1}{d-2}}$, $r \in (0,1)$, we derive an equation satisfied by $\tilde v$. By exploiting this equation and applying the theory of removable singularity, together with  the Harnack inequality in \cite{Ser-64}, we are able to show that $\tilde v(0) \in (0,\infty)$. This implies the asymptotic behavior near the origin in \eqref{asymp-behav}, which is a counterpart in the context of nonlocal nonlinearity of \cite[Theorem 1.2]{TraZog-15}. The decay rate in \eqref{asymp-behav} is obtained by using the argument in \cite{MorVan-2013}.

It is worth mentioning that the  uniqueness for \eqref{eq:EL-u} remains an open question. In the free-potential case (namely $\mu=0$), it is well-known that the kernel of the linearized operator of equation \eqref{eq:EL-u}  at a positive radial solution $w$ can be expressed as ${\rm span}\{\partial_{x_i}w: 1 \leq i \leq d\}$, which is a crucial ingredient in the proof of the uniqueness. However, such decomposition or its variants involving the Hardy operator $\cL_{\mu}$ is not known in the context of nonlocal nonlinearity, even for $\mu<\mu_0$. The presence of the Hardy potential also prevents the adaptation of the method in \cite{Lie-76,WanYi-17}. We stress that the ODE-based approach, which relies on the the generalized Pohoz\v aev identity in \cite{ShiWat-13} and was successfully applied for the NLS equation with power-type nonlinearity in \cite{MukNamNgu-2021}, is not applicable to problem \eqref{eq:NLS} due to the nonlocal nature of the Choquard term. 

The next theorem provides the existence of optimizers for the Hardy-Gagliardo-Nirenberg inequality involving the Choquard term.

\begin{theorem}[Hardy-Gagliardo-Nirenberg inequality] \label{thm:HGN} Let $d\ge 3$ and $\frac{d+\alpha}{d}<p<\frac{d+\alpha}{d-2}$. Then the following inequality holds 
\begin{equation} \label{eq:CHGN-def}
	\|\sqrt{\cL_{\mu_0}}u\|_{L^2}^{\theta} \|u\|_{L^2}^{1-\theta}  \geq {\sC}{\| (I_\alpha * |u|^p)|u|^p\|_{L^1}^{\frac{1}{2p}}}  \quad \text{where } \theta:= \frac{dp-(d+\alpha)}{2p} \in (0,1),
\end{equation}	
for all $u\in \cQ_{\mu_0}$, with the sharp constant  $\sC$ given by
\begin{equation} \label{C-W-def}
	{\sC} := \inf_{u\in \cQ_{\mu_0} \setminus \{0\}}\cW(u), \quad \text{where} \quad \cW(u) := \frac{ \|\sqrt{\cL_{\mu_0}}u\|_{L^2}^{\theta} \|u\|_{L^2}^{1-\theta} }{{\| (I_\alpha * |u|^p)|u|^p\|_{L^1}^{\frac{1}{2p}}}}.
\end{equation}
Moreover, $\sC$ is attained by a positive radial nonincreasing function $\tilde u \in \cQ_{\mu_0}$. 

In addition, any radial minimizer $u$ for \eqref{C-W-def} can be represented by
\begin{equation} \label{representation-minimizer}
u(x)=z Q(\lambda x)
\end{equation} 
for some $z\in \mathbb{C}$, $\lambda>0$ and a radial solution $Q \in \cQ_{\mu_0}$ to equation \eqref{eq:EL-u}. The sharp constant $\sC$ is related to $Q$ by
$$ \sC = \theta^{\frac{\theta}{2}}(1-\theta)^{\frac{1}{2p}-\frac{\theta}{2}}\| Q\|_{L^2}^{\frac{p-1}{p}}.
$$
\end{theorem}

Theorem \ref{thm:HGN}, together with Theorem \ref{th:groundstate}, extends \cite[Theorem 2.3]{FenYua-15} and \cite[Lemma 4.1 (1)]{Li-20} to the case of critical Hardy potential.
To prove the existence of an optimizer for \eqref{C-W-def}, we follow Weinstein' strategy, making use of compact embeddings in \cite{TraZog-15} (see Lemma \ref{lem:radial-compact}) and the Hardy-Littlewood-Sobolev inequality \eqref{inequa:HLS-3}.

A function $Q \in \cQ_{\mu_0}$ is called a \textit{ground state} if $Q$ is a positive radial solution to \eqref{eq:EL-u} such that $Q$ is a minimizer for problem \eqref{C-W-def}. The set of all ground states is denoted by $\cG$.

When $\frac{d+\alpha}{d}<p<\frac{d+\alpha}{d-2}$, thanks to the Pohoz\v aev identities \eqref{eq:Pohozaev-1} and \eqref{eq:Pohozaev-2},  if $Q \in \cG$ then 
\begin{align} \label{Mgs}
\| Q \|_{L^2} &= \theta^{-\frac{dp-d-\alpha}{4(p-1)}}(1-\theta)^{\frac{dp-d-\alpha-2}{4(p-1)}}\sC^{\frac{p}{p-1}}=:M_{\rm gs}, \\ \label{Kgs}
\| \sqrt{\cL_{\mu_0}}Q \|_{L^2} &= 	\theta^{\frac{1}{2}-\frac{dp-d-\alpha}{4(p-1)}}(1-\theta)^{\frac{dp-d-\alpha-2}{4(p-1)}-\frac{1}{2}}\sC^{\frac{p}{p-1}}=:H_{\rm gs}, \\ \label{Ngs}
\| (I_\alpha * Q^p)Q^p\|_{L^1} &= \theta^{-\frac{dp-d-\alpha}{2(p-1)}}(1-\theta)^{1+\frac{dp-d-\alpha-2}{2(p-1)}}\sC^{\frac{2p}{p-1}}=:N_{\rm gs}.
\end{align} 
It means that for any ground states $Q$, the quantities $\| Q \|_{L^2}$, $\| \sqrt{\cL_{\mu_0}}Q \|_{L^2}$ and $\| (I_\alpha * Q^p)Q^p\|_{L^1}$ are independent of $Q$. 

Let $E$ be the energy functional defined by
\begin{align} \label{energy}
	E(v) = \frac{1}{2}\| \sqrt{\cL}_{\mu_0}v \|_{L^2}^2 - G(v) \quad \text{with} \quad G(v) = \frac{1}{2p}\int_{\R^d}(I_\alpha * |v|^p)|v|^pdx, \quad v \in \cQ_{\mu_0}.
\end{align} 
Thanks to \eqref{Kgs} and \eqref{Ngs}, for any ground state $Q$, the energy $E(Q)$ is independent of $Q$ since
$$ E(Q) = \frac{1}{2}H_{\rm gs}^2 - \frac{1}{2p}N_{\rm gs}=: E_{\rm gs}.
$$

The results in Theorems \ref{th:groundstate} and \ref{thm:HGN} will be used to study problem \eqref{eq:NLS}. In order to state the main results, we introduce a notion of weak solutions to \eqref{eq:NLS}.

\begin{definition}[Weak solutions] A function $u$ is called a {\em weak solution} to \eqref{eq:NLS} in $(0,T)$ with initial datum 
	$u_0\in \cQ_{\mu_0}$  if 
	$$u \in C([0,T);\cQ_{\mu_0}) \cap C^1([0,T);\cQ_{\mu_0}^*)$$
	and it satisfies the Duhamel formula
	\begin{align} \label{eq:Duhamelu} 
		u(t)=e^{-it\cL_{\mu_0}}u_0 + i\int_0^t e^{-i(t-s)\cL_{\mu_0}} [(I_\alpha * |u|^p) |u|^{p-2}u](s)ds, \quad \forall t \in (0,T).
	\end{align}
Here $\cQ_{\mu_0}^*$ denotes the dual space of $\cQ_{\mu_0}$. In the above formula, the dependence on $x$ is omitted. If $T=\infty$ then $u$ is called a {\em global solution}. 
\end{definition}

When $p=2$, the local and global existence were obtained in  \cite[Remark 4.2]{Suz-16}, while blow-up criteria were showed in \cite[Theorem 4.1 (B3)]{Suz-17}. To the best of our knowledge, the local wellposedness of problem \eqref{eq:NLS} with $\frac{d+\alpha}{d} \leq p <2$ still remains open since both energy method \cite{Suz-13,Suz-16} and the standard fixed point argument are not applicable. In this paper, \textit{we consider problem  \eqref{eq:NLS} with} 
\begin{equation} \label{assump:p>2} d \geq 3, \quad (d-4)_+< \alpha < d \quad \text{and} \quad 2<p<\frac{d+\alpha}{d-2}. 
\end{equation}
It is known that under condition \eqref{assump:p>2}, for any $u_0 \in \cQ_{\mu_0}$, there is a constant $T=T(\| u_0\|_{\cQ_{\mu_0}})$ such that problem \eqref{eq:NLS} admits a unique weak solution $u \in C([0,T);\cQ_{\mu_0}) \cap C^1([0,T);\cQ_{\mu_0}^*)$. Moreover, the mass and energy are conserved, i.e. for any $t \in [0,T)$,
\begin{equation} \label{convervationlaw} \| u(t)\|_{L^2} = \| u_0 \|_{L^2}, \quad E(u(t)) = E(u_0).
\end{equation}
Furthermore, there is a maximum time $T^*$ in the sense that either $T^*=\infty$, namely the solution is global, or $T^*<\infty$ and $\lim_{t \uparrow T^*}\| \sqrt{\cL_{\mu_0}}u\|_{L^2} = \infty$, namely the solution blows up in finite time. See Proposition \ref{prop:localexistence-abs} for a detailed statement, which is derived from a more general result in \cite[Theorem 1.6]{Suz-20}. 

The next theorem provides sufficient conditions under which the global existence and finite-time blow-up for problem \eqref{eq:NLS} hold.
\begin{theorem} \label{th:global-blowup} Under assumption \eqref{assump:p>2}, let $u$ be the unique local weak solution with maximal lifespan $(0,T^*)$.
	
\noindent {\sc 1.  Sufficient conditions for global existence.} 
Assume one of the following cases holds

$(a)$ $d-2<\alpha<d$ and $2<p<\frac{d+\alpha+2}{d}$; 

$(b)$ $d-2<\alpha<d$, $p=\frac{d+\alpha+2}{d}$ and $\| u_0 \|_{L^2} < M_{\rm gs}$; 

$(c)$ $\max\{2,\frac{d+\alpha+2}{d}\} < p < \frac{d+\alpha}{d-2}$ and 
\begin{equation}\label{eqB3}
	E(u_0)\|u_0\|_{L^2}^\kappa<E_{\rm gs}M_{\rm gs}^\kappa \quad \text{and} \quad \|\sqrt{\mathcal{L}_{\mu_0}} u_0\|_{L^2}^2\|u_0\|_{L^2}^\kappa<H_{\rm gs}^2 M_{\rm gs}^\kappa.
\end{equation}
where 
\begin{equation} \label{kappa}
	\kappa=\frac{2(d+\alpha-(d-2)p)}{dp-d-\alpha-2}.
\end{equation}
Then $u$ exists globally, namely $T^*=\infty$. Moreover, in any of the above cases, $\| \sqrt{\cL_{\mu_0}}u(t) \|_{L^2}$ remains uniformly bounded in $t \in (0,\infty)$. \medskip

\noindent {\sc 2. Sufficient conditions for blow-up.} If $\frac{d+\alpha + 2}{d} \leq p <\frac{d+\alpha}{d-2}$, $|x| u_0 \in L^2(\R^d)$ and either $E(u_0)<0$, or  
\begin{equation} \label{cond-Eu-Q}
	E(u_0)\left\|u_0\right\|_{L^2}^\kappa<E_{\rm gs}M_{\rm gs}^\kappa, \quad\|\sqrt{\cL_{\mu_0}} u_0\|_{L^2}^2\|u_0\|_{L^2}^\kappa>H_{\rm gs}^2 M_{\rm gs}^\kappa,
\end{equation}
then $T^*<+\infty$.
\end{theorem}

Theorem \ref{th:global-blowup} is an extension of  \cite[Theorem 1.2 and Theorem 1.2 (i)]{FenYua-15} and \cite[Theorem 1.3 (ii), Theorem 1.4 and Theorem 1.6]{Li-20} to the context of critical Hardy potential. The proof of the global existence relies on the idea in \cite{Wei-83,KenMer-08}, while the finite-time blow-up follows from an adaptation of the methods in \cite{HolRou-07,KenMer-08,Suz-17}. For the sake of completeness, the proof of  Theorem \ref{th:global-blowup} is presented in Section \ref{sec:global-blowup}.

Finally, we provide a characterization of the minimal mass blow-up solutions in the spirit of the celebrated result in \cite{Mer-93}.

\begin{theorem}[Minimal mass blow-up solutions] \label{thm:min-mass-sol} Assume $d \geq 3$, $d-2<\alpha<d$ and $p=\frac{d+\alpha+2}{d}$. 

\noindent {\sc 1. Existence.} Let $\gamma\in \R$, $\lambda>0$, $T>0$ and $Q \in \cG$. Then 
		\begin{equation} \label{eq:u-0-T}
			u(t,x)= e^{i \gamma}e^{i\frac{\lambda^2}{T-t}}e^{-i\frac{|x|^2}{4(T-t)}}\left( \frac{\lambda}{T-t} \right)^{\frac{d}{2}} Q\left( \frac{\lambda x}{T-t} \right) \quad \forall x \in \R^d, \; t\in [0,T)
		\end{equation}
is a weak solution to problem \eqref{eq:NLS} in $(0,T)$ and blows up in the finite time $T$. 
		
\noindent {\sc 2. Characterization.} For any finite time $T>0$, if $u$ is a weak solution to problem \eqref{eq:NLS} in $(0,T)$ with $\|u_0\|_{L^2} = M_{{\rm gs}}$ and blows up at $T$, namely $\lim _{t \nearrow_T}\| \sqrt{\cL}_{\mu_0} u(t)\|_{L^2} = +\infty$,  then $u$ is given in \eqref{eq:u-0-T} for some constants  $\gamma\in \mathbb{R}$, $\lambda>0$ and $Q \in \cG$. 
\end{theorem}	

The proof of the characterization in Theorem  \ref{thm:min-mass-sol} relies on a compactness of minimizing sequence for \eqref{C-W-def}. In the case of subcritical Hardy potential, the compactness result is proved by using the equivalent between the seminorms $\| \nabla u \|_{L^2}$ and $\| \sqrt{\cL_{\mu}}u \|_{L^2}$ and the concentration-compactness principle expressed in terms of the linear profile decomposition (see, e.g., \cite[Proposition 4.3]{CheLuLu-19}). However, this analysis cannot be extended to the case of critical Hardy potential. In this case, we combine the concentration-compactness method \cite{Lio-84}, geometric-localization-based techniques,  and a decomposition of the Choquard nonlinearity, to obtain the compactness result. This allow us to adapt the strategy in \cite{HmiKer-05} to establish the characterization for blow-up solutions with critical mass. Our result extends and improve  \cite[Theorems 1.4 and 1.6]{FenYua-15} and \cite[Theorem 1.2 (2)]{CheLuLu-19}, and can be viewed as a counterpart in the context of nonlocal nonlinearities of \cite[Theorem 4]{MukNamNgu-2021}.

 \medskip
\noindent \textbf{Organization of the paper}. In Section \ref{sec:HGN}, we will prove the existence of optimizers for \eqref{C-W-def} and establish the Pohoz\v{a}ev identities, as well as the asymptotic behavior at the origin and at infinity for ground state solutions, which allow us to obtain Theorem \eqref{th:groundstate}. In Section \ref{sec:compactness}, we first establish a compactness lemma for minimizing sequences  for \eqref{C-W-def} and then prove Theorem \ref{thm:HGN}. Section \ref{sec:global-blowup} is devoted to the proof of global existence and finite-time blow-up in Theorem \ref{th:global-blowup}. In Section \ref{sec:minimalmass}, we show the characterization of minimal mass blow-up solutions stated in Theorem \ref{thm:min-mass-sol}. Finally, in the Appendix, we provides auxiliary results which will be used in the proof of the compactness result. \medskip

\noindent \textbf{Notations and conventions.} Throughout the paper, for simplicity, we write $\| \cdot\|_{L^p}$ for 	$\| \cdot\|_{L^p(\R^d)}$. For $r>0$ and $x \in \R^d$, we denote $B_r(x)=\{y \in \R^d: |y-x|<r\}$. If $x=0$, we simply write $B_r$ for $B_r(0)$. For $0<a<b$, we denote $A_{a,b}=\{x \in \R^d: a<|x|<b\}$. \medskip

\section{Hardy-Gagliardo-Nirenberg inequality and ground state solutions} \label{sec:HGN}
\subsection{Hardy-Gagliardo-Nirenberg inequality}
We start this section by recalling  the Hardy-Littlewood-Sobolev inequality (see e.g.  \cite[Theorem 4.3]{LieLos-01}).

\begin{lemma}[Hardy-Littlewood-Sobolev inequality]  Let $d \geq 1$, $0< \alpha < d$ and $q,r \in (1,\infty)$ such that $\frac{1}{q} + \frac{1}{r} = 1 + \frac{\alpha}{d}$. Then there exists a positive constant $C=C(d,\alpha,q)$ such that for any $f \in L^q(\R^d)$ and $g \in L^r(\R^d)$, 
\begin{equation} \label{inequa:HLS-3}
\left| \int_{\R^d} (I_\alpha * f)g \, dx \right| \leq C \| f \|_{L^q} \| g \|_{L^r}.
\end{equation}

\end{lemma}

Let $\cQ_{\mu_0,\rm rad}$ be the space of radial functions in $\cQ_{\mu_0}$. We quote the following compact embedding from \cite[Page 4999]{TraZog-15}.
\begin{lemma} \label{lem:radial-compact}
	For any $2<q<2^*$, the embedding $\cQ_{\mu_0,\rm rad} \hookrightarrow L^q(\R^d)$ is compact.
\end{lemma}

\begin{lemma} \label{lem:HGN} Let $d\ge 3$ and $\frac{d+\alpha}{d}<p<\frac{d+\alpha}{d-2}$. Then inequality \eqref{eq:CHGN-def} holds for all $u\in \cQ_{\mu_0}$. 
	Moreover, $\sC$ is attained by a positive radial non-increasing function $\tilde u \in \cQ_{\mu_0}$. In addition, $\tilde u$ can be represented by $\tilde u(x) = \nu_1 Q(\nu_2 x)$ for some $\nu_1,\nu_2>0$ and $Q \in \cQ_{\mu_0}$ is a nonnegative radial solution to equation \eqref{eq:EL-u}.
\end{lemma}

\begin{proof}

We first prove the existence of an optimizer for \eqref{eq:CHGN-def}. By using the Hardy-Littlewood-Sobolev inequality \eqref{inequa:HLS-3}, the embedding \eqref{eq:Q-in-Lp} and the standard scaling argument, we can show that $\sC>0$. 

We will employ Weinstein's strategy \cite{Wei-83} to show the existence of a minimizer of $\sC$. Since $C_0^\infty(\R^d \setminus \{0\})$ is dense in $\cQ_{\mu_0}$, we may assume that $\{u_n\} \subset C_0^\infty(\R^d \setminus \{0\})$ is a minimizing sequence for $\sC$ in \eqref{eq:CHGN-def}. By the P\'olya--Szeg\"o rearrangement inequality and the Riesz rearrangement inequality (see \cite[Theorem 3.4 and Theorem 3.7]{LieLos-01} and \cite{PolSze-51,Burchard-09}) we can assume that the functions $u_n$'s are non-negative and radially symmetric decreasing. Moreover, by a scaling argument we may also assume that
\begin{equation} \label{norm=1}
\|u_n\|_{L^2}=\| \sqrt{\cL_{\mu_0}} u_n\|_{L^2} =1, \quad  \lim_{n \to \infty}\|(I_\alpha * |u_n|^p)|u_n|^p\|_{L^1}^{\frac{1}{2p}} = \sC^{-1}.
\end{equation}
This implies that $\{u_n\}$ is uniformly bounded in $\cQ_{\mu_0}$. 
  
Thanks to the compact embedding in Lemma \ref{lem:radial-compact}, for any $q \in (2,2^*)$, $u_n\to \tilde u$ strongly in  $L^q(\R^d)$  as $n \to \infty$.
In particular, since $2<\frac{2dp}{d+\alpha}<2^*$, $u_n\to \tilde u$ strongly in  $L^{\frac{2dp}{d+\alpha} }(\R^d)$  as  $n \to \infty$. Consequently, 
\begin{equation} \label{converp}|u_n|^p \to |\tilde u|^p \quad \text{ strongly in } L^{\frac{2d}{d+\alpha} }(\R^d) \text{ as } n \to \infty.
\end{equation} 
By the Hardy-Littlewood-Sobolev inequality and H\"older's inequality,
\begin{align*}
&\left| \int_{\R^d} (I_\alpha * |u_n|^p)|u_n|^p dx - \int_{\R^d} (I_\alpha * |\tilde u|^p)|\tilde u|^p dx \right| \\
&\lesssim \left( \int_{\R^d}|u_n|^{\frac{2dp}{d+\alpha} } dx \right)^{\frac{d+\alpha}{2d}} \left(  \int_{\R^d} ||u_n|^p - |\tilde u|^p|^{\frac{2d}{d+\alpha} }dx \right)^{\frac{d+\alpha}{2d} }  \\
&\quad + \left( \int_{\R^d}||u_n|^p - |\tilde u|^p|^{\frac{2d}{d+\alpha} } dx \right)^{\frac{d+\alpha}{2d}} \left(  \int_{\R^d} |\tilde u|^{\frac{2dp}{d+\alpha} }dx \right)^{\frac{d+\alpha}{2d} }. 
\end{align*}
Therefore, by using \eqref{converp}, we derive
\begin{equation} \label{convergeIa}
\|(I_\alpha * |\tilde u|^p)|\tilde u|^p \|_{L^1}^{\frac{1}{2p}}  = \lim_{n \to \infty}\|(I_\alpha * |u_n|^p)|u_n|^p \|_{L^1}^{\frac{1}{2p} } = \sC^{-1}.
\end{equation}

On the other hand, by \eqref{norm=1} and the Banach-Alaoglu theorem, up to a subsequence as $n\to \infty$ again, we can assume that 
$$u_n\wto \tilde u, \quad \sqrt{\cL_{\mu_0}} u_n \wto \sqrt{\cL_{\mu_0}} \tilde u \quad \text{ weakly in $L^2(\R^d)$}, $$
which lead to
\begin{equation} \label{limuL} \|\tilde u\|_{L^2}\le \liminf_{n\to \infty} \|u_n\|_{L^2}=1, \quad \|\sqrt{\cL_{\mu_0}} \tilde u\|_{L^2}\leq \liminf_{n\to \infty} \|\sqrt{\cL_{\mu_0}} u_n\|_{L^2}=1.
\end{equation}
Combining \eqref{convergeIa} and \eqref{limuL} yields
$$
\frac{ \|\sqrt{\cL_{\mu_0}}\tilde u\|_{L^2(\R^d)}^{\theta} \|\tilde u\|_{L^2(\R^d)}^{1-\theta} }{\|(I_\alpha * |\tilde u|^p)|\tilde u|^p\|_{L^1}^{\frac{1}{2p}}} \leq  \sC. 
$$
This implies that $\tilde u$ is a minimizer for the variational problem \eqref{eq:CHGN-def}. In view of the above proof, since the minimizing sequence $\{u_n\}$ are nonnegative radially symmetric decreasing, the limit $\tilde u$ is also nonnegative radially symmetric decreasing.

Next we derive the Euler-Lagrange equation. By using standard variational techniques and the constraints that
$$
\|\tilde u\|_{L^2} =\| \sqrt{\cL_{\mu_0}} \tilde u\|_{L^2} = 1,  \quad \|(I_\alpha * \tilde u^p)\tilde u^p \|_{L^1}^{\frac{1}{2p}}  = \sC^{-1},
$$
we can show that $\tilde u$ satisfies the Euler-Lagrange equation
\begin{equation} \label{eq:EL-non-normalized}
\theta \cL_{\mu_0} \tilde u + (1-\theta) \tilde u - \sC^{2p} (I_\alpha * \tilde u^p) \tilde u^{p-1}=0.
\end{equation}

For $\nu_1,\nu_2>0$, put 
$$\tilde u(x)=\nu_1 Q (\nu_2 x), \quad x \in \R^d \setminus \{0\},
$$
then it is easy to check that
$$ \cL_{\mu_0} \tilde u(x) =  \nu_1 \nu_2^2 \cL_{\mu_0} Q(\nu_2 x), \quad [(I_\alpha * \tilde u^p)\tilde u^{p-1}](x) = \nu_1^{2p-1}\nu_2^{-\alpha} [(I_\alpha * Q^p) Q^{p-1}](\nu_2 x).
$$
Plugging the above relations into \eqref{eq:EL-non-normalized} leads to
\begin{equation} \label{eq:Q-1}
\cL_{\mu_0}Q + \nu_2^{-2}\theta^{-1}(1-\theta) Q - \sC^{2p}\nu_1^{2p-2}  \nu_2^{-\alpha-2}  \theta^{-1}(I_\alpha * Q^p) Q^{p-1} = 0.
\end{equation}
By choosing
$$\nu_1= \theta^{-\frac{\alpha}{4(p-1)} } (1-\theta)^{\frac{\alpha+2}{4(p-1)} }\sC^{-\frac{p}{p-1} }   \quad \text{and} \quad  \nu_2 = \left( \frac{1-\theta}{\theta} \right)^{\frac{1}{2} },
$$
we infer from \eqref{eq:Q-1} that $Q$ solves \eqref{eq:EL-u}.
Moreover, $Q$ is nonnegative radially symmetric decreasing and 
\begin{equation} \label{eq:CHGN-Q}
\|Q\|_{L^2} =  \nu_1^{-1} \nu_2^{\frac{d}{2}} \|\tilde u\|_{L^2} = \theta^{\frac{\alpha-d(p-1)}{4(p-1)} } (1-\theta)^{\frac{d(p-1)-\alpha-2}{4(p-1)} }\sC^{\frac{p}{p-1} }.
\end{equation}
The proof is complete.
\end{proof}

\subsection{Pohoz\v aev identities}
We will establish Poho\v zaev identities by using the ground state representation and a techniques based on cutoff functions.  
\begin{proposition}  \label{prop:Pohozaev} 
Let $p>1$. Assume $w \in \cQ_{\mu_0} \cap  L^{\frac{2dp}{d+\alpha}}(\R^d) \cap W_{\mathrm{loc}}^{2,q}(\R^d \setminus \{0\})$, for some $q>1$, is a solution to equation \eqref{eq:EL-u} in the pointwise sense.
 Then the following identities hold
\begin{align} \label{eq:Pohozaev-1}
&\|w\|_{L^2}^2 = \frac{d+\alpha- (d-2)p}{2p} \| (I_\alpha*|w|^p)	|w|^p \|_{L^1} \\ \label{eq:Pohozaev-2}
&\| \sqrt{\cL_{\mu_0}}w\|_{L^2}^2 = \frac{dp-d-\alpha}{2p} \| (I_\alpha*|w|^p)	|w|^p \|_{L^1}.
\end{align}
\end{proposition}

\begin{proof}
Put $v(x)=|x|^{\frac{d-2}{2}}w(x)$ for $x \neq 0$. Then $v$ satisfies
\begin{equation} \label{eq:transform-v} 
- \mathrm{div}(|x|^{-(d-2)}\nabla v) +|x|^{-(d-2)}v = |x|^{-\frac{(d-2)p}{2}}[I_\alpha*(|\cdot|^{-\frac{(d-2)p}{2}}|v|^p)]|v|^{p-2}v 
\end{equation}
a.e. in $\R^d \setminus \{0\}$.
Since $w \in \cQ_{\mu_0}$, it follows that $v \in H^1(\R^d;|x|^{-(d-2)})$ and
\begin{align} \label{v-L2}
\int_{\R^d} |x|^{-(d-2)}v^2 dx &= \int_{\R^d} w^2 dx, 
\int_{\R^d} |x|^{-(d-2)}|\nabla v|^2 dx &= \int_{\R^d} \left( |\nabla w|^2 - \frac{\mu_0}{|x|^2}w^2\right)dx.	
\end{align}

Let $\phi \in C_c^\infty(\R^d)$ such that $0 \leq \phi \leq 1$, $\phi =1$ in $B_1$, $\mathrm{supp} \, \phi \subset B_2$ and let $\varphi \in C^\infty(\R^d)$ such that $0 \leq \varphi \leq 1$, $\varphi =1$ in $B_2^c$ and $\mathrm{supp} \, \varphi \subset B_1^c$. For any $k \in \N$ and $\epsilon >0$, put 
$$\phi_k(x) = \phi_k(\frac{x}{k}), \quad \varphi_\epsilon(x) = \varphi(\frac{x}{\epsilon}).$$ 
Multiplying equation \eqref{eq:transform-v} with $\phi_k \varphi_{\epsilon}(\nabla v \cdot x)$ and integrating over $\R^d$ yield
\begin{equation} \label{eq:multi-v-1} \begin{aligned}
&-\int_{\R^d}\mathrm{div}(|x|^{-(d-2)}\nabla v(x)) \phi_k \varphi_{\epsilon} (\nabla v \cdot x)dx +\int_{\R^d}(|x|^{-(d-2)}v) \phi_k \varphi_{\epsilon} (\nabla v \cdot x)dx 
\\\quad &= \int_{\R^d}|x|^{-\frac{(d-2)p}{2}}[I_\alpha*(|\cdot|^{-\frac{(d-2)p}{2}}|v(\cdot)|^p)](|v|^{p-2}v) \phi_k \varphi_{\epsilon}	(\nabla v \cdot x)dx.	
\end{aligned}
\end{equation}
We will estimate the terms in \eqref{eq:multi-v-1} successively. \medskip

\noindent \textbf{Step 1:} We deal with the first term in the left hand side of \eqref{eq:multi-v-1}.

By using integration by parts and straightforward computations, we obtain
\begin{equation} \label{eq:Poho-1a} \begin{aligned}
&- \int_{\R^d}\mathrm{div}(|x|^{-(d-2)}\nabla v) \phi_k \varphi_{\epsilon} (\nabla v \cdot x)dx \\
&= \int_{\R^d} |x|^{-(d-2)} \phi_k \varphi_{\epsilon} \nabla v \cdot  \nabla (\nabla v \cdot x)dx + \int_{\R^d} |x|^{-(d-2)} \phi_k \nabla \varphi_{\epsilon} \cdot \nabla v (\nabla v \cdot x)dx\\
&\quad + \int_{\R^d} |x|^{-(d-2)}  \varphi_{\epsilon} \nabla \phi_k \cdot  \nabla v (\nabla v \cdot x) dx.
\end{aligned} \end{equation} 
 We employ  the identities $\nabla v \cdot \nabla (\nabla v \cdot x) = \frac{1}{2}\nabla |\nabla v|^2\cdot x + |\nabla v|^2$ and $\operatorname{div} (|x|^{-(d-2)}x) = 2|x|^{-(d-2)}$, together with  integration by parts,  to derive
\begin{align*}
\int_{\R^d} |x|^{-(d-2)} \phi_k \varphi_{\epsilon} \nabla v \cdot \nabla (\nabla v \cdot x)dx &=	- \frac{1}{2}\int_{\R^d} |x|^{-(d-2)}(\nabla \phi_k \cdot x) \varphi_{\epsilon} |\nabla v|^2 dx \\
&\quad - \frac{1}{2}\int_{\R^d} |x|^{-(d-2)} \phi_k (\nabla \varphi_{\epsilon}\cdot x) |\nabla v|^2 dx, 
\end{align*}
which, combined with \eqref{eq:Poho-1a}, implies
\begin{equation} \label{eq:Poho-1} \begin{aligned}
		&- \int_{\R^d}\mathrm{div}(|x|^{-(d-2)}\nabla v) \phi_k \varphi_{\epsilon} (\nabla v \cdot x)dx \\
		&= - \frac{1}{2} \int_{A_{k,2k}} |x|^{-(d-2)}(\nabla \phi_k \cdot x) \varphi_{\epsilon} |\nabla v|^2 dx +  \int_{A_{k,2k}} |x|^{-(d-2)} \varphi_{\epsilon} \nabla \phi_k \cdot \nabla v (\nabla v \cdot x) dx\\ 
		&\quad - \frac{1}{2}\int_{A_{\epsilon,2\epsilon}} |x|^{-(d-2)} \phi_k (\nabla \varphi_{\epsilon}\cdot x) |\nabla v|^2 dx
		 + \int_{A_{\epsilon,2\epsilon}} |x|^{-(d-2)} \phi_k \nabla \varphi_{\epsilon} \cdot \nabla v (\nabla v \cdot x) dx.
\end{aligned} \end{equation}
By the dominated convergence theorem, together with \eqref{v-L2} and estimate $|\nabla \phi_k|\lesssim k^{-1}$ and $|\nabla \varphi_{\epsilon}| \lesssim \epsilon^{-1}$, we can let $k \to \infty$ and $\epsilon \to 0$ in \eqref{eq:Poho-1} to derive
\begin{equation} \label{eq:Poho-2}
\lim_{\epsilon \to 0, \, k \to \infty}\int_{\R^d}\mathrm{div}(|x|^{-(d-2)}\nabla v) \phi_k \varphi_{\epsilon} (\nabla v \cdot x)dx = 0.
\end{equation} 
\textbf{Step 2:} We deal with the second term in the left hand side of \eqref{eq:multi-v-1}. By integration by parts, we have
\begin{equation} \label{eq:Poho-3} \begin{aligned}
&\int_{\R^d} |x|^{-(d-2)}v \phi_k \varphi_{\epsilon} (\nabla v \cdot x)dx \\ 
&= - \int_{\R^d} |x|^{-(d-2)} \phi_k \varphi_{\epsilon} v^2 dx - \frac{1}{2} \int_{A_{k,2k}}|x|^{-(d-2)} (\nabla \phi_k \cdot x)\varphi_{\epsilon} v^2 dx \\
&\quad - \frac{1}{2} \int_{A_{\epsilon,2\epsilon}}|x|^{-(d-2)} \phi_k (\nabla \varphi_{\epsilon} \cdot x) v^2 dx,
\end{aligned}\end{equation}
where $A_{a,b}=\{x \in \R^d: a<|x|<b\}$.
By the dominated convergence theorem and \eqref{v-L2}, we can let $k \to \infty$ and $\epsilon \to 0$ in \eqref{eq:Poho-3} to derive
\begin{equation} \label{eq:Poho-4}
	\lim_{\epsilon \to 0, \, k \to \infty}\int_{\R^d} |x|^{-(d-2)}v \phi_k \varphi_{\epsilon} (\nabla v \cdot x)dx  = 	- \int_{\R^d} |x|^{-(d-2)}v^2 dx.
\end{equation} 
\textbf{Step 3:} We deal with the term in the right hand side of \eqref{eq:multi-v-1}. By integration by parts and using the identities $$\mathrm{div}(|x|^{-\frac{(d-2)p}{2}}x) = \frac{2d-(d-2)p}{2}|x|^{-\frac{(d-2)p}{2}} \quad \text{and} \quad \nabla_x I_\alpha(x-y)= \frac{-(d-\alpha)(x-y)}{|x-y|^{d-\alpha+2}}, 
$$
we obtain
\begin{align*}  
&\int_{\R^d}|x|^{-\frac{(d-2)p}{2}}[I_\alpha*(|\cdot|^{-\frac{(d-2)p}{2}}|v(\cdot)|^p)](x)(|v(x)|^{p-2}v(x))	\phi_k(x) \varphi_{\epsilon}(x) (\nabla v(x) \cdot x)dx \\ 
&= \frac{1}{2p} \int_{\R^d}|x|^{-\frac{(d-2)p}{2}}\left(\int_{\R^d} I_\alpha(x-y)|y|^{-\frac{(d-2)p}{2}}|v(y)|^p dy\right) \phi_k(x) \varphi_{\epsilon}(x)	\nabla_x (|v(x)|^p) \cdot x \, dx \\
&\quad + \frac{1}{2p} \int_{\R^d}|y|^{-\frac{(d-2)p}{2}}\left(\int_{\R^d} I_\alpha(x-y)|x|^{-\frac{(d-2)p}{2}}|v(x)|^p dx\right) \phi_k(y) \varphi_{\epsilon}(y)	\nabla_y (|v(y)|^p) \cdot y \, dy \\
&= \frac{(d-2)p-2d}{2p} \int_{\R^d}|x|^{-\frac{(d-2)p}{2}} [I_\alpha*(|\cdot|^{-\frac{(d-2)p}{2}}|v(\cdot)|^p](x) \phi_k(x) \varphi_{\epsilon}(x)	|v(x)|^p  dx \\ 
&\quad +\frac{d-\alpha}{2p}\iint_{\R^d \times \R^d}|x|^{-\frac{(d-2)p}{2}}|v(x)|^p |y|^{-\frac{(d-2)p}{2}}|v(y)|^p \frac{(x-y)(x \phi_k(x) \varphi_{\epsilon}(x) - y \phi_k(y) \varphi_{\epsilon}(y)) }{|x-y|^{d-\alpha+2}}dxdy \\ 
&\quad  - \frac{1}{p} \int_{\R^d}|x|^{-\frac{(d-2)p}{2}} [I_\alpha*(|\cdot|^{-\frac{(d-2)p}{2}}|v(\cdot)|^p](x) \phi_k(x) (\nabla \varphi_{\epsilon}(x) \cdot x ) |v(x)|^p dx \\ 
&\quad- \frac{1}{p} \int_{\R^d}|x|^{-\frac{(d-2)p}{2}} [I_\alpha*(|\cdot|^{-\frac{(d-2)p}{2}}|v(\cdot)|^p](x) (\nabla \phi_k(x) \cdot x)  \varphi_{\epsilon}(x) |v(x)|^p dx.
\end{align*}
Therefore, by the dominated convergence theorem, we obtain
\begin{equation} \label{eq:Poho-5} \begin{aligned}
&\lim_{\epsilon \to 0,\, k \to \infty}\int_{\R^d}|x|^{-\frac{(d-2)p}{2}}[I_\alpha*(|\cdot|^{-\frac{(d-2)p}{2}}|v(\cdot)|^p)](|v|^{p-2}v)	\phi_k \varphi_{\epsilon} (\nabla v \cdot x)dx \\
&=	\frac{(d-2)p-d-\alpha}{2p} \int_{\R^d}|x|^{-\frac{(d-2)p}{2}} [I_\alpha*(|\cdot|^{-\frac{(d-2)p}{2}}|v(\cdot)|^p](x)	|v(x)|^p  dx.
\end{aligned} \end{equation}

Letting $k \to \infty$ and $\epsilon \to 0$ successively in \eqref{eq:Poho-1} and taking into account \eqref{eq:Poho-2}, \eqref{eq:Poho-4} and \eqref{eq:Poho-5}, we deduce
$$ - \int_{\R^d} |x|^{-(d-2)}v^2 dx = \frac{(d-2)p-d-\alpha}{2p} \int_{\R^d}|x|^{-\frac{(d-2)p}{2}} [I_\alpha*(|\cdot|^{-\frac{(d-2)p}{2}}|v(\cdot)|^p]	|v|^p  dx.
$$
This implies \eqref{eq:Pohozaev-1} due to the relation $w(x) = |x|^{-\frac{d-2}{2}}v(x)$ for $x \neq 0$. 

Next, by multiplying equation \eqref{eq:EL-u} by $w$ and integrating over $\R^d$, we get
\begin{equation} \label{eq:Poho-6}
\int_{\R^d} \left(|\nabla w|^2 - \frac{\mu_0}{|x|^2} w^2 \right)dx + \int_{\R^d}w^2 dx = \int_{\R^d} (I_\alpha*|w|^p)	|w|^p  dx.	
\end{equation}
Combining \eqref{eq:Pohozaev-1} and \eqref{eq:Poho-6} yields \eqref{eq:Pohozaev-2}. The proof is complete.
\end{proof}

\subsection{Blow-up and decay estimates}

The Moser iteration is a standard tool for establishing $L^\infty$ regularity for positive solutions to a wide class of elliptic equations. In the presence of a Hardy potential, we employ a ground state representation to transform \eqref{eq:EL-u} into a divergence-form equation with singular coefficients, to which the Moser iteration can be applied. As a result, we obtain an upper bound for solutions to \eqref{eq:EL-u} on $\R^d \setminus \{0\}$, which is sharp near the origin. This strategy was used in e.g. \cite{DMPS-16}.

\begin{proposition} \label{prop:asymp-phi0-near0} Assume conditions in  \eqref{assump:p>2} hold. Let $w \in \cQ_{\mu_0}$ be a positive solution to \eqref{eq:EL-u}. Then $w \in C^2(\R^d \setminus \{0\})$ and
	\begin{align} \label{lim0-phi0-1}
		w(x) \leq C|x|^{-\frac{d-2}{2}} \quad \forall x \in \R^d \setminus \{0\}.
	\end{align}	 
\end{proposition}
\begin{proof}
	Put $v(x) = |x|^{\frac{d-2}{2}}w(x)$ for $x \in \R^d \setminus \{0\}$, then $v \in H^1(\R^d;|x|^{-(d-2)})$ satisfies \eqref{eq:transform-v}. 
	For $\beta \geq 1$ and $\lambda>0$, define
	\begin{equation}
		\Phi_{\beta,\lambda}(\ell) := 
		\left\{ \begin{aligned}
			&\ell^{\beta} \quad &&\text{if } 0 \leq \ell \leq \lambda, \\
			&\beta \lambda^{\beta-1}(\ell-\lambda) + \lambda^\beta &&\text{if } \ell > \lambda.
		\end{aligned} \right.	
	\end{equation}
	Note that $\Phi_{\beta,\lambda}$ is a Lipschitz (with Lipschitz constant $\beta \lambda^{\beta-1}$) convex function with $\Phi_{\beta,\lambda}(0) = 0$, hence $\Phi_{\beta,\lambda}(v) \in H^1(\R^d;|x|^{-(d-2)})$. Moreover,
	\begin{equation} \label{eq:nablaPhi} - \mathrm{div}(|\cdot|^{-(d-2)} \nabla \Phi_{\beta,\lambda}(v)) \leq - \Phi_{\beta,\lambda}'(v)   \mathrm{div}(|\cdot|^{-(d-2)} \nabla v)  \quad \text{weakly in } H^1(\R^d;|x|^{-(d-2)}).
	\end{equation}
	
Since $\cQ_{\mu_0} \hookrightarrow L^{\frac{2dp}{d+\alpha}}(\R^d)$ (due to $2<\frac{2dp}{d+\alpha}<2^*$), we derive
$$
\|  |x|^{-\frac{d-2}{2}}v\|_{L^{\frac{2dp}{d+\alpha}}} \leq C\| v \|_{H^1(\R^d;|x|^{-(d-2)})}, \quad \forall v \in H^1(\R^d;|x|^{-(d-2)}).
$$
Combining the above estimate with the inequalities $\Phi_{\beta,\lambda}'(v)v - \Phi_{\beta,\lambda}(v) \geq 0$, $\Phi_{\beta,\lambda}'(\ell)\ell \leq \beta \Phi_{\beta,\lambda}(\ell)$, the Hardy-Littlewood-Sobolev inequality in \eqref{inequa:HLS-3} and \eqref{eq:nablaPhi},
	we obtain
	\begin{align} \nonumber
		&\| |x|^{-\frac{d-2}{2}}\Phi_{\beta,\lambda}(v)\|_{L^{\frac{2dp}{d+\alpha}}}^2 \\ \nonumber
		&\leq C\int_{\R^d} (-\Phi_{\beta,\lambda}(v)\Phi_{\beta,\lambda}'(v) \mathrm{div}(|x|^{-(d-2)} \nabla v) + \Phi_{\beta,\lambda}(v)^2 |x|^{-(d-2)}) dx \\ \nonumber
		&\leq C\beta \int_{\R^d} |x|^{-\frac{(d-2)p}{2}} \Phi_{\beta,\lambda}^2(v)  (I_\alpha * (|\cdot|^{-\frac{d-2}{2}}v )^p)v^{p-2} dx \\ \nonumber
		&\leq C\beta \| |\cdot|^{-\frac{d-2}{2}}v)^p  \|_{L^{\frac{2d}{d+\alpha} }} \| |\cdot|^{-\frac{(d-2)p}{2}} \Phi_{\beta,\lambda}^2(v) v^{p-2} \|_{L^{\frac{2d}{d+\alpha} }} \\ \label{est:Phi-1}
		&\leq C\beta \|v \|_{H^1(\R^d;|x|^{-(d-2)})}^p  \| |\cdot|^{-\frac{(d-2)p}{2}} \Phi_{\beta,\lambda}^2(v) v^{p-2} \|_{L^{\frac{2d}{d+\alpha} }}.
	\end{align}

\noindent 
Let $M>0$ to be determined later on. By H\"older's inequality and the fact that $\Phi_{\beta,\lambda}(v) \leq v^{\beta}$, we have
	\begin{align} \nonumber
		&\| |x|^{-\frac{(d-2)p}{2}} \Phi_{\beta,\lambda}^2(v) v^{p-2} \|_{L^{\frac{2d}{d+\alpha} }} \\ \nonumber
		&\leq \| \1_{\{v \leq M\}}|x|^{-\frac{(d-2)p}{2}} \Phi_{\beta,\lambda}^2(v) v^{p-2} \|_{L^{\frac{2d}{d+\alpha} }} + \| \1_{\{v \geq M\}}|x|^{-\frac{(d-2)p}{2}} \Phi_{\beta,\lambda}^2(v) v^{p-2} \|_{L^{\frac{2d}{d+\alpha} }} \\ \nonumber
		&\leq M^{p-2} \| \1_{\{v \leq M\}}|x|^{-\frac{(d-2)p}{2}} \Phi_{\beta,\lambda}^2(v)  \|_{L^{\frac{2d}{d+\alpha} }} \\ \label{est:Phi-2}
		&\quad + \| \1_{\{v \geq M\}}|x|^{-\frac{d-2}{2}}\Phi_{\beta,\lambda}(v) \|_{L^{\frac{2dp}{d+\alpha}}}^2 \| \1_{\{v \geq M\}}|x|^{-\frac{d-2}{2}}v\|_{L^{\frac{2dp}{d+\alpha}}}^{p-2}.
	\end{align}
We can find $M$ large enough such that
	$$ C\beta \| |\cdot|^{-\frac{(d-2)}{2}}v \|_{H^1(\R^d;|x|^{-(d-2)})}^p \| \1_{\{v \geq M\}}|x|^{-\frac{d-2}{2}}v\|_{L^{\frac{2dp}{d+\alpha}}}^{p-2} \leq \frac{1}{2},
	$$
	and then fix $M$. 
	Combining \eqref{est:Phi-1}, \eqref{est:Phi-2} and the estimate $\Phi_{\beta,\lambda}(v) \leq v^{\beta}$ yields
	\begin{align} \label{est:iteration-1}
		\| |x|^{-\frac{d-2}{2}}\Phi_{\beta,\lambda}(v)\|_{L^{\frac{2dp}{d+\alpha}}}^2 &\leq c_0\beta \| |x|^{-\frac{(d-2)p}{2}} v^{2\beta}  \|_{L^{\frac{2d}{d+\alpha} }},
	\end{align}
	where the constant $c_0$ in the second estimate depends on fixed $M$. By Fatou's lemma, we derive
	\begin{equation} \label{v-iterate}
	\| v \|_{L^{\frac{2dp\beta}{d+\alpha}}(\R^d;|x|^{-\frac{(d-2)dp}{d+\alpha}})} \leq (c_0\beta)^{\frac{1}{2\beta}}\| v \|_{L^{\frac{4d\beta}{d+\alpha}}(\R^d;|x|^{-\frac{(d-2)dp}{d+\alpha}})}.
	\end{equation}
	We will use Moser's iteration on the above estimate to obtain $L^\infty$-bound for $v$. For $k \geq 1$, let us define 
	$$ \beta_k:=\left(\frac{p}{2}\right)^k, \quad A_k:= \| v \|_{L^{\frac{2dp\beta_k}{d+\alpha}}(\R^d;|x|^{-\frac{(d-2)dp}{d+\alpha}})}, \quad   k \geq 0.
	$$
	Then we infer from \eqref{v-iterate} that
	$$ A_{k+1} \leq (c_0\beta_{k+1})^{\frac{1}{2\beta_{k+1}}}A_k, \quad k \geq 0,
	$$
which implies
$$ A_{k+1} \leq \left( c_0^{\sum_{j=1}^k\frac{1}{2\beta_i}} \right)\left( \Pi_{i=1}^k \beta_j^{\frac{1}{2\beta_j}} \right)A_0.
$$
Using the fact that $c_0^{\sum_{j=1}^\infty\frac{1}{2\beta_i}}<\infty$,  $\Pi_{i=1}^\infty \beta_j^{\frac{1}{2\beta_j}}<\infty$, and $A_0 = \|  |x|^{-\frac{d-2}{2}}v\|_{L^{\frac{2dp}{d+\alpha}}} \lesssim \| v \|_{H^1(\R^d;|x|^{-(d-2)})}$, we derive that
	$$
	A_{k+1} \leq C_* \| v \|_{H^1(\R^d;|x|^{-(d-2)})}.
	$$
	Now fix $R>0$, then we derive from the above estimate and the definition of $A_{k+1}$ that
	$$
	R^{-\frac{d-2}{\beta_{k+1}}}\left( \int_{B_R} v^{\beta_{k+1}q} dx \right)^{\frac{1}{\beta_{k+1} q}} \leq C_* \| v \|_{H^1(\R^d;|x|^{-(d-2)})}.
	$$
	Letting $k \to \infty$ and $R \to \infty$ successively yields
	$$ \| v \|_{L^\infty} \leq C_* \| v \|_{H^1(\R^d;|x|^{-(d-2)})},
	$$
	which in turn implies \eqref{lim0-phi0-1}.
	
	Next, we derive from \eqref{lim0-phi0-1} that for any $K \Subset \R^d \setminus \{0\}$, $w, |x|^{-2}w, (I_\alpha* w^p)w^{p-1} \in L^\infty(K)$. Therefore, we infer from equation \eqref{eq:EL-u} that 
	$$ - \Delta w = -w + \mu_0 |x|^{-2}w + (I_\alpha * w^p)w^{p-1} \in L^\infty(K).
	$$
	By the standard regularity for elliptic equations, we derive that $w \in W^{2,q}(K')$ for any $K' \Subset K$ and $1 \leq q \leq \infty$. Using the Sobolev embeddings, we obtain $w \in C^1_{\textrm{loc}}(\R^d \setminus \{0\})$. Next, by arguing along the same line as in the proof of \cite[Claim 3, Proposition 4.1]{MorVan-13}, we deduce that $w \in C^2(\R^d \setminus \{0\})$. The proof is complete.
\end{proof}

\begin{proposition} \label{prop:asympt0} Assume conditions in  \eqref{assump:p>2} hold. Let $w \in \cQ_{\mu_0}$ be a positive radial solution to \eqref{eq:EL-u}. There holds
	\begin{equation} \label{finitelimit}
		\lim_{|x| \to 0}|x|^{\frac{d-2}{2}}w(x) \in (0,\infty).	
	\end{equation}	
\end{proposition}
\begin{proof} Using \eqref{lim0-phi0-1}, $w \in \cQ_{\mu_0} \hookrightarrow L^{\frac{2dp}{d+\alpha}}(\R^d)$ and the assumption $p<\frac{d+\alpha}{d-2}$, we obtain, for any $x \in B_1 \setminus \{0\}$, that 
\begin{align*}
(I_\alpha * w^p)(x) &= \int_{\{|x-y| \leq \frac{1}{2}|x|\}} |x-y|^{-(d-\alpha)}w(y)^p dy + 	\int_{\{|x-y| \geq \frac{1}{2}|x|\}} |x-y|^{-(d-\alpha)}w(y)^p dy \\
&\lesssim |x|^{-\frac{(d-2)p}{2}+\alpha} + |x|^{-\frac{d-\alpha}{2}} \lesssim |x|^{-\frac{d-\alpha}{2}}.
\end{align*}	 
Therefore
\begin{equation} \label{est:nonlinear-near0} (I_\alpha * w^p)(x)w^{p-2}(x) \lesssim |x|^{-\frac{d-\alpha+(d-2)(p-2)}{2}}, \quad \forall x \in B_1 \setminus \{0\}.
\end{equation}
Since $w$ is a positive radial solution to \eqref{eq:EL-u}, we have
$$ -w_{rr} - \frac{d-1}{r}w_r - \frac{(d-2)^2}{4r^2}w = (I_\alpha * w^p)w^{p-1} \quad \text{in } \R^d \setminus \{0\}.
$$ 
Set $v(x)=|x|^{\frac{d-2}{2}}w(x)$ for $x \neq 0$, then $v \in C^2(\R^d \setminus \{0\})$ and $v$ satisfies
$$ v_{rr} + \frac{1}{r}v_r = a(r)v \quad \text{in } \R^d \setminus \{0\},
$$
where $a$ is a radial function defined by
$$ a(x): = (I_\alpha * (|\cdot|^{-\frac{d-2}{2}}v)^p)(x)(|x|^{-\frac{d-2}{2}}v(x))^{p-2} - 1.
$$
	
We define
	$$ \tilde v(\rho) := v(r), \quad \tilde a(\rho) := a(r) \quad \text{with } \rho=(-\ln(r))^{-\frac{1}{d-2}}, \, r \in (0,1).
	$$
	Then 
	\begin{equation} \label{eq:tildev-radial} \tilde v_{\rho\rho} + \frac{d-1}{\rho}\tilde v_{\rho} + \tilde V \tilde v = 0, \quad \rho \in (0,1),
	\end{equation}
	where $ \tilde V(\rho): = (d-2)^2 \rho^{-2(d-1)}e^{-2\rho^{-(d-2)}}\tilde a(\rho)$.
Combining \eqref{est:nonlinear-near0}, the definition of $a$, $\tilde a$, $\tilde V$ and the relation $r=e^{-\rho^{-(d-2)}}$, we deduce
$$ \tilde V(\rho) \lesssim \rho^{-2(d-1)}e^{-(2-\frac{d-\alpha + (d-2)(p-2)}{2})\rho^{-(d-2)}} \lesssim 1, \quad \forall \rho \in (0,1).
$$
In the above second estimate, we have used the fact that  $2-\frac{d-\alpha + (d-2)(p-2)}{2}>0$ due to $p<\frac{d+\alpha}{d-2}$. 

By \eqref{eq:tildev-radial}, we see that $\tilde v$ is a solution to
	\begin{equation} \label{eq:tildeV-B1} - \Delta \tilde v + \tilde V \tilde v = 0
	\end{equation}
in $B_1 \setminus \{0\}$. Note that $|\nabla \tilde v| \in L^2(B_1)$ due to the fact that $v \in H^1(\R^d;|x|^{-(d-2)})$ and $\tilde v \in L^\infty(B_1) \subset L^2(B_1)$; hence $\tilde v \in H^1(B_1)$. Therefore, employing \cite[Corollary 1 on page 176]{Ser-64}, we derive that $0$ is a removable singularity and $\tilde v$ can be extended as a continuous solution of equation \eqref{eq:tildeV-B1} in the whole domain $B_1$. Moreover, by the Harnack inequality \cite[Theorem 5]{Ser-64}, we deduce that $\tilde v(0)>0$. This implies the desired result.  
\end{proof}

Next we show that ground state solutions decay when $|x| \to \infty$.

\begin{proposition} \label{prop:asym-infty}
Assume conditions in  \eqref{assump:p>2} hold. Let $w \in \cQ_{\mu_0}$ be a positive solution to \eqref{eq:EL-u}. 
	
	(i) Then there exists $R_0>0$ large enough such that 
	\begin{equation} \label{est:decayQ}
		w(x) \leq C|x|^{-\frac{d-1}{2}}e^{-\frac{1}{2}|x|}	\quad \forall |x|>R_0.
	\end{equation}
	
	(ii) Assume in addition that $w$ is radial, then
	\begin{equation} \label{est:limdecayQ} \lim_{|x| \to \infty}|x|^{\frac{d-1}{2}}e^{|x|}w(x) \in (0,\infty).
	\end{equation}
\end{proposition}
\begin{proof}
(i)	By using \eqref{lim0-phi0-1}, the fact that $w \in \cQ_{\mu_0}$ and the assumption $p>2$, we derive that
	$$ (I_\alpha * w^p)(x)w^{p-2}(x) \leq C|x|^{-\frac{d-2}{2(d+\alpha)}[(d-\alpha)p+(d+\alpha)(p-2)]}, \quad \forall x \in \R^d \setminus B_4. 
	$$	
	Therefore, there exists $R=R(p) > 4$ such that
	$$ \frac{\mu_0}{|x|^2}w(x) +  (I_\alpha * w^p)(x)w^{p-1}(x) \leq \frac{3}{4}w(x), \quad \forall x \in \R^d \setminus B_R. 
	$$
	As a consequence, $w$ satisfies  $-\Delta w + \frac{1}{4}w \leq 0$ in $\R^d \setminus B_{R}$.
	Let $\varphi \in C^2(\R^d \setminus B_R)$ be the solution to
	\begin{equation} \label{eq:linear} 
		\left\{  \begin{aligned}
			&-\Delta \varphi + \frac{1}{4}\varphi = 0 \quad \text{ in } \R^d \setminus B_{R}, \\
			&\varphi = CR^{-\frac{d-2}{2}}  \text{ on } \partial B_{R}, \quad 
			\lim_{|x| \to \infty} \varphi(x) = 0,
		\end{aligned} \right. \end{equation}
where $C$ is the constant in \eqref{lim0-phi0-1}. By \cite[Lemma 6.4]{MorVan-2013}, we have
$$ \varphi(x) \leq C_1|x|^{-\frac{d-1}{2}}e^{-\frac{|x|}{2}}, \quad \forall x \in \R^d \setminus B_R.
$$ 	 
On the other hand, note that $w \leq \varphi$ on $\partial B_R$ due to \eqref{lim0-phi0-1}, hence by the comparison principle for elliptic equation, we derive $ w \leq \varphi$ in  $\R^d \setminus B_R$. 
Consequently, we obtain \eqref{est:decayQ}. 

(ii) From \eqref{est:decayQ}, we can show that
$$ (I_\alpha * w^p)(x)w^{p-2}(x) \leq C_2|x|^{-(\frac{d-\alpha}{2} + (d-1)(p-2))}e^{-\frac{(p-2)|x|}{2}} \leq C_2 e^{-\frac{(p-2)|x|}{2}}, \quad \forall x \in \R^d \setminus B_R. 
$$	
Therefore $ -\Delta w + w \geq 0 \geq - \Delta w + Ww$ in $\R^d \setminus B_R$,
where
$$ W(x) = 1- \mu_0|x|^{-2} - C_2 e^{-\frac{(p-2)|x|}{2}}, \quad x \in \R^d \setminus B_R.
$$
By using the same argument as in the proof of \cite[Proposition 6.3]{MorVan-2013}, we obtain \eqref{est:limdecayQ}. 
\end{proof}

\begin{proof}[\textbf{Proof of Theorem \ref{th:groundstate}}]
The existence part in statement (1) is obtained in Lemma \ref{lem:HGN}. When $\frac{d+\alpha}{d}<p<\frac{d+\alpha}{d-2}$, if $w \in \cQ_{\mu_0}$, by using the standard bootstrap argument, together with the embedding \eqref{eq:Q-in-Lp} and the standard elliptic regularity theory, we can show that $w \in L^{\frac{2dp}{d+\alpha}(\R^d)} \cap W_{\rm loc}^{2,q}(\R^d \setminus \{0\})$ for some $q>1$. Therefore, estimates \eqref{eq:Pohozaev-1}, \eqref{eq:Pohozaev-2} follow from Proposition \ref{prop:Pohozaev}.

Statement (2) follows easily from Proposition \ref{prop:Pohozaev} too.  Indeed, suppose by contradiction that $w \in \cQ_{\mu_0} \cap  L^{\frac{2dp}{d+\alpha}}(\R^d) \cap W_{\mathrm{loc}}^{2,q}(\R^d)$ is a nontrivial solution to equation \eqref{eq:EL-u}. If $p \geq \frac{d+\alpha}{d-2}$ then the left-hand side of \eqref{eq:Pohozaev-1} is strictly negative, while the right-hand side is nonnegative, which is impossible. Similarly, if $1<p \leq \frac{d+\alpha}{d}$, we reach to a contradiction due to \eqref{eq:Pohozaev-2}.  

Statement (3) is derived from Propositions \ref{prop:asymp-phi0-near0},  \ref{prop:asympt0} and \ref{prop:asym-infty}. 	The proof is complete.
\end{proof}

\section{Compactness of minimizing sequences and representation for minimizers} \label{sec:compactness}
This section is devoted to the proof of a compactness result for minimizing sequences for \eqref{C-W-def} and its application to showing the representation \eqref{representation-minimizer}.
\subsection{Compactness of minimizing sequences}
In this subsection, we will show that any (normalized) minimizing sequences (not necessarily radial) for \eqref{C-W-def} is pre-compact. The proof includes several steps and combines the concentration-compactness method \cite{Lio-84} and refined localization techniques to handle both the critical Hardy potential and the Choquard nonlinearity. More specifically, we will rule out the vanishing case, which enables us to extract locally convergent part whose limit is non-zero. Then we use the IMS formula to decompose the kinetic energy and perform a careful analysis based on the Hardy-Littlewood-Sobolev inequality to decompose the energy induced by the Choquard term. Consequently, we deduce that $\{u_n\}$ is convergent in $L^2(\R^d)$, which implies the desired convergence due to the interpolation.

\begin{theorem}[Compactness of minimizing sequences] \label{precompact} Let $d\ge 3$, $(d-4)_+<\alpha<d$ and $2<p<\frac{d+\alpha}{d-2}$. Assume that $\{u_n\}$ is a minimizing sequence for \eqref{C-W-def}   such that 
	$$
	\liminf_{n\to \infty} \|u_n\|_{L^2}>0, \quad \limsup_{n\to \infty} \| u_n\|_{\cQ_{\mu_0}}<\infty. 
	$$
	Then, up to a subsequence, $\{u_n\}$ converges to a function $\tilde u$  in $\cQ_{\mu_0}$. Consequently, $\tilde u$ is a minimizer for \eqref{C-W-def}.
\end{theorem}

\begin{proof} By a suitable scaling, we may assume that 
	$$
	\|u_n\|_{L^2}= \|\sqrt{\cL_{\mu_0}} u_n\|_{L^2}= 1, \quad  \left(\int_{\R^d}(I_\alpha * |u_n|^p)|u_n|^p dx \right)^{\frac{1}{2p}}  \to \sC^{-1}.
	$$
	Since $\{u_n\}$ is bounded in $\cQ_{\mu_0}$, by \eqref{eq:Q-in-Lp}, up to a subsequence, $u_n \wto \tilde u$  in $\cQ_{\mu_0}$ and in $H^s(\R^d)$ for all $0<s<1$. Moreover, by the Sobolev compact embedding, for any $2 \leq q< 2^*$ and $R>0$,  $\1_{B_R}u_n \to \1_{B_R} \tilde u$ in $L^q(\R^d)$. In particular, $u_n  \to \tilde u$ a.e. in $\R^d \setminus \{0\}$.	By Fatou's lemma, we have $\| \tilde u\|_{L^2} \leq 1$. \medskip
	
	\noindent \textbf{Step 1:} No vanishing. We  assert that the vanishing case of $\{|u_n|^2\}_{n \in \N}$ cannot not occur, namely for any $R>0$,
	\begin{equation}\label{eq:novanishing}
		\limsup _{n \to \infty} \sup _{y \in \R^d} \int_{B_R(y)} |u_n|^2 d x > 0. 
	\end{equation}
	Indeed, by contradiction, we suppose that \eqref{eq:novanishing} does not hold. By Lemma \ref{lem:vanishing}, we deduce that $\|u_n\|_{L^q} \to 0$ as $n \to \infty$ for any $2< q <2^*$. By the  Hardy-Littlewood-Sobolev inequality and noting that $2 < \frac{2dp}{d+\alpha}< 2^*$, we derive
		$$ 0 < \sC^{2p} \leftarrow \|(I_\alpha * |u_n|^p) |u_n|^p\|_{L^1}  \leq C\| u_n \|_{L^{\frac{2dp}{d+\alpha}}}^{2p} \to 0,
		$$
	which is a contradiction. Therefore, \eqref{eq:novanishing} hold true for any $R>0$. \medskip

	\noindent \textbf{Step 2:} Extracting the locally convergent part. From \eqref{eq:novanishing}, up to a subsequence
	$$
	 		\lim _{n \to \infty} \sup _{y \in \R^d} \int_{B_1(y)} |u_n|^2 d x > \delta>0. 
	 $$
	Therefore, there exists a subsequence of $\{u_n\}$, still denoted by the same notation, and a sequence $\{y_n\} \subset \R^d$ such that $\int_{B_1(y_n)}|u_n|^2dx > \frac{\delta}{2}$ for all  $n \in \N$.
Let $\{R_n\}$ be an increasing sequence such that $R_n \to +\infty$ and $B_1(y_n) \subset B_{R_n}$ for any $n \in \N$. Then 
\begin{equation}\label{eq:novanishing-B1}
\int_{B_{R_n}}|u_n|^2dx  \geq \int_{B_1(y_n)}|u_n|^2dx > \frac{\delta}{2}, \quad \forall n \in \N.
\end{equation}
	
	We will prove that $\|\tilde u\|_{L^2}=1$. By contradiction, we suppose that $\|\tilde u\|_{L^2}<1$. Then from Lemma \ref{lem:extract} and \eqref{eq:novanishing-B1}, up to a subsequence, 
	\begin{equation} \label{convergence-localize}
		\lim_{n \to \infty}\| \1_{A_{R_n,3R_n}}u_n\|_{L^2} = 0, \quad \tilde \ell:=\lim_{n \to \infty}\|\1_{B_{2R_n}}u_n\|_{L^2}  \in (0,1),
	\end{equation}
	where $A_{a,b}=\{x \in \R^d: a < |x|<b\}$ for $0<a<b$. \medskip
	
	\noindent \textbf{Step 3:} Decomposition of the Hardy term. Let $\chi, \eta \in C^\infty(\R^d)$ such that $\mathrm{supp} \, \chi \subset B_1$, $\chi=0$ in $B_2^c$ and  $\chi^2+\eta^2=1$ in $\R^d$. For any $n \in \N$, set
	$$
	\chi_n (x) = \chi\big(\frac{x}{R_n}\big), \quad \eta_n(x) = \eta\big(\frac{x}{R_n}\big), \quad  x \in \R^d.
	$$
	By virtue of the IMS formula (see \cite{LieYau-88}), we have the decomposition
	\begin{align}\label{eq:localization-0} \nn
		\| \sqrt{\cL_{\mu_0}}u_n \|_{L^2}^2 &= \| \sqrt{\cL_{\mu_0}}(\chi_n u_n) \|_{L^2}^2 + \| \nabla (\eta_n u_n)\|_{L^2}^2  - \mu_0 \left\| \frac{\eta_n u_n}{|x|}\right\|_{L^2}^2 \\
		&\quad  - \| u_n \nabla \chi_n \|_{L^2}^2 - \| u_n \nabla \eta_n \|_{L^2}^2  \nn\\
		&= \| \sqrt{\cL_{\mu_0}}(\chi_n u_n) \|_{L^2}^2 + \| \nabla (\eta_n u_n)\|_{L^2}^2 + o(1)_{n\to \infty}.
	\end{align}
	Then by \eqref{eq:CHGN-def} and the Gagliardo-Nirenberg inequality without Hardy potential (see, e.g., \cite[Theorem 2.3]{FenYua-15}), we have
	\begin{align} \nn
		&\| \sqrt{\cL_{\mu_0}}u_n \|_{L^2}^{2\theta} \|u_n \|_{L^2}^{2(1-\theta)} \\	\nn
		&\geq \left( \| \sqrt{\cL_{\mu_0}}(\chi_n u_n) \|_{L^2}^2 + \| \nabla (\eta_n u_n)\|_{L^2}^2 + o(1)_{n\to \infty} \right)^{\theta}\left( \| \chi_n u_n\|_{L^2}^2 +  \| \eta_n u_n\|_{L^2}^2 \right)^{1-\theta}  \\ \nn
		&\geq \| \sqrt{\cL_{\mu_0}}(\chi_n u_n) \|_{L^2}^{2\theta} \| \chi_n u_n\|_{L^2}^{2(1-\theta)} + \| \nabla (\eta_n u_n) \|_{L^2}^{2\theta} \| \eta_n u_n\|_{L^2}^{2(1-\theta)} + o(1)_{n\to \infty} \\ \label{est:decompose-oper-1}
		&\geq \sC^2 \| (I_\alpha * |\chi_n  u_n|^p)|\chi_n u_n|^p \|_{L^1}^{\frac{1}{p}} + \sC_0^2 \| (I_\alpha * |\eta_n  u_n|^p)|\eta_n u_n|^p \|_{L^1}^{\frac{1}{p}} + o(1)_{n\to \infty}.
	\end{align}
	where
	$$
	\sC_0 :=\inf_{\varphi \in H^1(\R^d) \setminus \{0\}} \frac{ \|\nabla  \varphi\|_{L^2}^\theta \|\varphi \|_{L^2}^{1-\theta} }{\|(I_\alpha * |\eta_n  u_n|^p)|\eta_n u_n|^p \|_{L^1}^{\frac{1}{2p}}}.
	$$
	It is well-known that $\sC_0$ has a minimizer (see, e.g., \cite[Theorem 2.3]{FenYua-15}), say $v_0 \in H^1(\R^N)$. Then we have
	$$
	\sC_0 = \frac{ \|\nabla v_0\|_{L^2}^\theta \|v_0 \|_{L^2}^{1-\theta} }{\|(I_\alpha * |v_0|^p)|v_0|^p \|_{L^1}^{\frac{1}{2p}}} > \frac{ \|\sqrt{\cL_{\mu_0}} v_0\|_{L^2}^\theta \|v_0 \|_{L^2}^{1-\theta} }{\|(I_\alpha * |v_0|^p)|v_0|^p \|_{L^1}^{\frac{1}{2p}}} \geq \sC.
	$$
	Let $\delta_0>0$ such that $\sC_0^2 =\sC^2 + \delta_0^2>0$.

	\noindent \textbf{Step 4:} Decomposition of the Choquard nonlinearity. We will show that
	\begin{equation} \label{intdecomp-1}	  \begin{aligned} 
			\left(\int_{\R^d} (I_\alpha * |u_n|^p) |u_n|^p dx \right)^{\frac{1}{p}} \leq & \left(\int_{\R^d} (I_\alpha * |\chi_n u_n|^p) |\chi_n u_n|^p dx \right)^{\frac{1}{p} } \\
			&+  \left(\int_{\R^d} (I_\alpha * |\eta_n u_n|^p) |\eta_n u_n|^p dx \right)^{\frac{1}{p} }  + o(1)_{n \to \infty}. 
		\end{aligned}
	\end{equation}	
	
	From the first convergence in \eqref{convergence-localize} and the embedding \eqref{eq:Q-in-Lp}, we have for any $2 \leq q < 2^*$, 
			\begin{equation} \label{converge-anu} \lim_{n \to \infty}\| \1_{A_{R_n,3R_n}}u_n\|_{L^q} = 0. 
			\end{equation}
			
	Put
	$ \varphi_n: = u_n^p - (\chi_n u_n)^p - (\eta_n u_n)^p$.  \medskip
	
\noindent \textit{Case 1:} $p \geq 2$. In this case,  $\varphi_n = u_n^p - (\chi_n^p + \eta_n^p)u_n^p \geq u_n^p - (\chi_n^2 + \eta_n^2)u_n^p = 0$. 
	By using the elementary inequality
	$$ 0 \leq (a+b)^\kappa - a^\kappa - b^\kappa \lesssim_{\kappa} a^{\kappa-1}b + ab^{\kappa-1}, \quad \forall a>0,b>0, \kappa \geq 1,
	$$
	and the fact that $0 \leq \chi_n, \eta_n \leq 1$, $\chi_n^2 + \eta_n^2 = 1$, we obtain
	\begin{align}  \label{varphin<}  
		0 \leq \varphi_n &= |\chi_n^2 u_n + \eta_n^2 u_n|^p - |\chi_n u_n|^p - |\eta_n u_n|^p.
	\end{align}
	We write
	\begin{align}
		&\int_{\R^d} (I_\alpha * |u_n|^p) |u_n|^p dx 	\nn \\
		&= \int_{\R^d} [I_\alpha * ( |\chi_n u_n|^p + |\eta_n u_n|^p + \varphi_n) ]( |\chi_n u_n|^p + |\eta_n u_n|^p + \varphi_n) dx \nn \\ 
		&= \int_{\R^d} (I_\alpha * |\chi_n u_n|^p) |\chi_n u_n|^p dx +  \int_{\R^d} (I_\alpha * |\eta_n u_n|^p) |\eta_n u_n|^p dx  + \sum_{j=1}^6 J_{j,n}, \label{SigmaJ}
	\end{align}
	where
	\begin{align*} 
		&J_{1,n}:=2\int_{\R^d} (I_\alpha * |\chi_n u_n|^p) |\eta_n u_n|^p dx, \quad &&J_{2,n}:=\int_{\R^d} (I_\alpha * |\chi_n u_n|^p) \varphi_n dx,	\\
		&J_{3,n}:=\int_{\R^d} (I_\alpha * |\eta_n u_n|^p) \varphi_n dx,	
		&&J_{4,n}:=\int_{\R^d} (I_\alpha * \varphi_n) |\chi_n u_n|^p dx, \\  &J_{5,n}:=\int_{\R^d} (I_\alpha * \varphi_n) |\eta_n u_n|^p dx,	
		&&J_{6,n}:= \int_{\R^d} (I_\alpha * \varphi_n) \varphi_n dx. 
	\end{align*}
	Let $\epsilon>0$ small enough such that $\alpha+\epsilon<d$. By the Hardy-Littlewood-Sobolev inequality \eqref{inequa:HLS-3}, the continuous embedding $\cQ_{\mu_0} \hookrightarrow L^{\frac{2dp}{d+\alpha}}(\R^d) \cap L^{\frac{2dp}{d+\alpha+\epsilon}}(\R^d)$, we estimate $J_{1,n}$ as
	\begin{align*}
		J_{1,n} &\lesssim R_n^{-\epsilon}\int_{\R^d}\1_{B_{3R_n}^c}(x)|u_n(x)|^p  \int_{\R^d}	\frac{\1_{B_{2R_n}}(y)|u_n(y)|^p}{|x-y|^{d-\alpha-\epsilon}}dy dx \\
		&\quad +\int_{\R^d}\1_{A_{R_n,3R_n}}(x)|u_n(x)|^p  \int_{\R^d}	\frac{\1_{B_{2R_n}}(y)|u_n(y)|^p}{|x-y|^{d-\alpha}}dy dx \\
		&\lesssim R_n^{-\epsilon}\| \1_{B_{3R_n}^c}u_n \|_{L^{\frac{2dp}{d+\alpha+\epsilon}}} \| \1_{B_{2R_n}}u_n \|_{L^{\frac{2dp}{d+\alpha+\epsilon}}} + \| \1_{A_{R_n,3R_n}}u_n \|_{L^{\frac{2dp}{d+\alpha}}} \| \1_{B_{2R_n}}u_n \|_{L^{\frac{2dp}{d+\alpha}}} \\
		&\lesssim R_n^{-\epsilon}\| u_n \|_{\cQ_{\mu_0}}^2  + \| \1_{A_{R_n,3R_n}}u_n \|_{L^{\frac{2dp}{d+\alpha}}} \| u_n \|_{\cQ_{\mu_0}}.
	\end{align*}	
	Therefore, since $2<\frac{2dp}{d+\alpha}<2_*$, thanks to the convergence \eqref{converge-anu}, $\lim_{n \to \infty}J_{1,n} = 0$. 
	
	Next, we use \eqref{varphin<}  to estimate $J_{2,n}$ as follows 
	\begin{align*} J_{2,n} &\lesssim \int_{\R^d} (I_\alpha * |\chi_n u_n|^p) \1_{A_{R_n,2R_n}} |u_n|^p dx \lesssim  \| u_n\|_{L^{\frac{2dp}{d+\alpha} }}^p \| \1_{A_{R_n,2R_n} } u_n \|_{L^{\frac{2dp}{d+\alpha} }}^p,
	\end{align*}
	which, together with the convergence \eqref{converge-anu}, implies $\lim_{n \to \infty}J_{2,n} = 0$. Similarly, we can show that $\lim_{n \to \infty}J_{j,n} = 0$ for all \quad $3 \leq j \leq 6$. Therefore, from \eqref{SigmaJ}, we infer \eqref {intdecomp-1}. \medskip
	
	\noindent \textit{Case 2:} $\frac{d+\alpha}{d}<p < 2$. Then by using the elementary estimate 
	$$(a+b)^\kappa \leq a^\kappa+b^\kappa, \quad \forall a>0, b>0,  0<\kappa<1,
	$$ and estimate \eqref{inequa:HLS-3}, we have
	\begin{equation} \label{intdecomp-1a} \begin{aligned}
	\int_{\R^d} (I_\alpha * |u_n|^p) |u_n|^p dx 
			&\leq \int_{\R^d} (I_\alpha * |\chi_n u_n|^p) |\chi_n u_n|^p dx + \int_{\R^d} (I_\alpha * |\eta_n u_n|^p) |\eta_n u_n|^p dx \\
			&+ \int_{\R^d} (I_\alpha * |\chi_n u_n|^p) |\eta_n u_n|^p dx + \int_{\R^d} (I_\alpha * |\eta_n u_n|^p) |\chi_n u_n|^p dx.		
		\end{aligned}
	\end{equation}
	As above, we can show that
	$$ \lim_{n \to \infty}\int_{\R^d} (I_\alpha * |\chi_n u_n|^p) |\eta_n u_n|^p dx = \lim_{n \to \infty} \int_{\R^d} (I_\alpha * |\eta_n u_n|^p) |\chi_n u_n|^p dx = 0.
	$$
	Therefore, from \eqref{intdecomp-1a}, we also derive \eqref{intdecomp-1}.	\medskip
	

\noindent \textbf{Step 5:} End of the proof.  Combining \eqref{est:decompose-oper-1} and \eqref{intdecomp-1} yields
	\begin{equation} \label{est:decompose-oper-2}
		\| \sqrt{\cL_{\mu_0}}u_n \|_{L^2}^{2\theta} \|u_n \|_{L^2}^{2(1-\theta)} \\
		\geq \sC^2 \| (I_\alpha * |u_n|^p)|u_n|^p \|_{L^1}^{\frac{1}{p}} + \delta_0^2 \| (I_\alpha * |\eta_n  u_n|^p)|\eta_n u_n|^p \|_{L^1}^{\frac{1}{p}} + o(1)_{n\to \infty}.
	\end{equation}
	Since $\{u_n\}$ is a minimizing sequence for \eqref{eq:CHGN-def}, we infer from \eqref{est:decompose-oper-2} that
	$$
	\lim_{n \to \infty }\| (I_\alpha * |\eta_n  u_n|^p)|\eta_n u_n|^p \|_{L^1}^{\frac{1}{2p}}=0.
	$$ 
	Plugging it back into \eqref{intdecomp-1} leads to
	\begin{equation} \label{est:decompose-Choquard-4}
		\| (I_\alpha * |u_n|^p)|u_n|^p \|_{L^1}^{\frac{1}{p}} \leq \| (I_\alpha * |\chi_n  u_n|^p)|\chi_n u_n|^p \|_{L^1}^{\frac{1}{p}}  + o(1)_{n\to \infty}.
	\end{equation}
	Now, by combining \eqref{eq:localization-0}, \eqref{est:decompose-Choquard-4} and the fact that $\lim_{n \to \infty}\| \chi_n u_n\|_{L^2}=\tilde \ell < 1 = \| u_n \|_{L^2}$, we obtain
	\begin{align} \nn
		\| \sqrt{\cL_{\mu_0}}u_n \|_{L^2}^{2\theta} \|u_n \|_{L^2}^{2(1-\theta)} 
		&\geq \left( \| \sqrt{\cL_{\mu_0}}(\chi_n u_n) \|_{L^2} + o(1)_{n\to \infty} \right)^{2\theta}\left( \frac{\| \chi_n u_n\|_{L^2} + o(1)_{n\to \infty}}{\tilde \ell} \right)^{2(1-\theta)}  \\ \nn
		\label{est:decompose-oper-1}
		&\geq \tilde \ell^{2(\theta-1)}\sC^2 \| (I_\alpha * |\chi_n  u_n|^p)|\chi_n u_n|^p \|_{L^1}^{\frac{1}{p}}  + o(1)_{n\to \infty} \\ \nn
		&\geq \tilde \ell^{2(\theta-1)}\sC^2 \| (I_\alpha * | u_n|^p)|u_n|^p \|_{L^1}^{\frac{1}{p}}  + o(1)_{n\to \infty}.
	\end{align}
	It follows that
	$$
	\lim_{n \to \infty}\frac{\| \sqrt{\cL_{\mu_0}}u_n \|_{L^2}^{\theta} \|u_n \|_{L^2}^{1-\theta}}{\| (I_\alpha * | u_n|^p)|u_n|^p \|_{L^1}^{\frac{1}{2p}}} \geq \tilde \ell^{\theta-1} \sC > \sC,
	$$
	where we have use the fact that $\tilde \ell, \theta \in (0,1)$. Therefore  we reach a contradiction as  $\{u_n\}$ is a minimizing sequence for  \eqref{eq:CHGN-def}. Thus we have proved that $\| \tilde u\|_{L^2}=1$.
	
	Now we are going to show that $\tilde u$ is a minimizer of \eqref{eq:CHGN-def}. Since $u_n  \wto u$ weakly in $L^2(\R^d)$, we derive that $u_n\to \tilde u$ in $L^2(\R^d)$, and hence $u_n\to \tilde u$ in $L^q(\R^d)$ by interpolation, for any $2 \leq q <2^*$.  Thus we conclude that $\tilde u$ is a minimizer and  $u_n\to \tilde u$ strongly in $\cQ_{\mu_0}$. 
\end{proof}

\subsection{Minimizers for the Hardy-Gagliardo-Nirenberg inequality} 
In this subsection, we will use Lemma Lemma \ref{lem:HGN} and Theorem \ref{precompact} to prove \ref{thm:HGN}.
 
\begin{proof}[\textbf{Proof of Theorem \ref{thm:HGN}}] The existence of a minimizer for \eqref{eq:CHGN-def} is proved in Lemma \ref{lem:HGN}.
	
Now let $u \in \cQ_{\mu_0}$ be a radial minimizer for \eqref{C-W-def}. Since $C_0^\infty(\R^d \setminus \{0\})$ is dense in $\cQ_{\mu_0}$, there exists a sequence $\{u_n\} \subset C_0^\infty(\R^d \setminus \{0\})$ such that $u_n \to u$ in $\cQ_{\mu_0}$. Consequently, $\sqrt{\cL_{\mu_0}}u_n \to \sqrt{\cL_{\mu_0}}u_n$, $u_n \to u$ in $L^2(\R^d)$ and $\|(I_\alpha * |u_n|^p)|u_n|^{p-2}u_n\|_{L^1} \to \|(I_\alpha * |u|^p)|u|^{p-2}u\|_{L^1}$ thanks to the Hardy-Littlewood-Sobolev inequality and H\"older's inequality. Therefore $\{u_n\}$ is a minimizing sequence for \eqref{C-W-def}. By standard argument based on the diamagnetic inequality, we deduce that $\{|u_n|\}$ is a minimizing sequence and hence by Theorem \ref{precompact}, we derive that $|u|$ is minimizer for \eqref{C-W-def}. By a suitable scaling as in the proof of Lemma \ref{lem:HGN}, $|u|$ can be written as $|u(x)|= \mu Q(\lambda x)$ for some $\mu>0,\lambda>0$ and $Q \in \cG$.

Put $\eta(x) = u(x)/|u(x)|$. Since $|\eta|^2=1$, we have $\Re(\bar \eta \nabla \eta)=0$. Hence
$$ \| \sqrt{\cL_{\mu_0}}u\|_{L^2}^2 = \| \sqrt{\cL_{\mu_0}}|u|\|_{L^2}^2 + \int_{\R^d}|\nabla \eta|^2|u|^2dx.
$$ 
If $|\nabla \eta| \not \equiv 0$ then $\sC=\cW(|u|) < \cW(u) = \sC$, which is a contradiction. Therefore,  $|\nabla \eta|  \equiv 0$ and hence $\eta$ is a constant. Thus $u(x)=zQ(\lambda x)$for $z \in \C$, $\lambda>0$ and $Q \in \cG$ and the proof is complete. 
\end{proof}

\section{Global existence and finite-time blow-up} \label{sec:global-blowup}
\subsection{Global Existence}
We first recall a local existence result for \eqref{eq:NLS}.
\begin{proposition} \label{prop:localexistence-abs}
\noindent {\em (1) (Local existence and conservation laws)}	Assume $(d-4)_+ < \alpha < d$ and $2<p<\frac{d+\alpha}{d-2}$. Then for any $u_0\in \cQ_{\mu_0}$, there
	exists $T=T(\| u_0 \|_{\cQ_{\mu_0}})$ such that \eqref{eq:NLS}
	admits
	a unique local weak solution $u\in C([0,T], \cQ_{\mu_0}) \cap C^1([0,T];\cQ_{\mu_0}^*)$. Moreover, the mass and energy conservation laws hold, namely 
	\begin{equation} \label{conservation-law}
	\|u(t)\|_{L^2}=\|u_0\|_{L^2}, \quad E(u(t))=E(u_0) \text{ for any }  t \in [0,T],
	\end{equation}
where $E(v)$ and $G(v)$ are defined in \eqref{energy}. \medskip

\noindent {\em (2) (Unique continuation and blow-up alternative)} The solution $u$ admits a unique continuation up to a maximum time $T^*$ in the sense that either $T^*=\infty$, namely the solution is global, or $T^*<\infty$ and $\lim_{t \uparrow T^*}\| \sqrt{\cL_{\mu_0}}u\|_{L^2} = \infty$, namely the solution blows up in finite time.	
\end{proposition}
Note that  the local existence, uniqueness and conservation laws  for problem \eqref{eq:NLS} were obtained in \cite[Theorem 1.6]{Suz-20} (even for more general type of potential and nonlinearity), while the unique continuation and blow-up alternative can be proved by arguing along the same line as in \cite[page 29]{MukNamNgu-2021}.

To prove the global existence for  \eqref{eq:NLS}, we adapt the standard arguments in \cite{Wei-83,KenMer-08} to the framework with Hardy potential and Choquard nonlinearity. The proof is presented below.

 \label{subsec:global}
\begin{proof}[\textbf{Proof of part (1) of Theorem \ref{th:global-blowup}}]
 By the blow-up alternative, it is sufficient to show that $\|\sqrt{\mathcal{L}_{\mu_0}} u(t)\|_{L^2}$ remains bounded uniformly in $t$. By the mass and energy conservation,  estimate \eqref{eq:CHGN-def} for any $t \in [0,T)$, 
\begin{equation}\label{eqB1}
\begin{aligned}
E\left(u_0\right)=E(u(t)) & =\frac{1}{2}\|\sqrt{\mathcal{L}_{\mu_0}} u(t)\|_{L^2}^2-\frac{1}{2p} \int_{\R^d} (I_\alpha * |u(t)|^p)|u(t)|^p dx \\
& \geq \frac{1}{2}\|\sqrt{\mathcal{L}_{\mu_0}} u(t)\|_{L^2}^2-\frac{1}{2p \sC^{2p}}\|\sqrt{\mathcal{L}_{\mu_0}} u(t)\|_{L^2}^{2p \theta}\left\|u_0\right\|_{L^2}^{2p(1-\theta)}.
\end{aligned}
\end{equation}
We will consider separably 3 cases. \medskip

\noindent (i) \textbf{Case 1:} $d-2<\alpha$ and $2<p<\frac{d+\alpha+2}{d}$. In this case $p \theta<1$, therefore, by Young's inequality, for any $\varepsilon>0$ small, we obtain
$$
\|\sqrt{\mathcal{L}_{\mu_0}} u(t)\|_{L^2}^{2p\theta} \leq \varepsilon\|\sqrt{\mathcal{L}_{\mu_0}} u(t)\|_{L^2}^2+C_{\varepsilon} .
$$
Inserting this into \eqref{eqB1} implies that $\|\sqrt{\mathcal{L}_{\mu_0}} u(t)\|_{L^2}$ is uniformly bounded in $t>0$. \medskip

\noindent (ii)  \textbf{Case 2:} $d-2<\alpha$ and $p=\frac{d+\alpha+2}{d}$. In this case, $p \theta=1$ and the sharp constant in (\ref{eq:CHGN-Q}) satisfies
$ p \sC^{2p}=M_{{\rm gs}}^{2(p-1)}$. 
Plugging this into  \eqref{eqB1} yields
$$
E\left(u_0\right)=E(u(t)) \geq \frac{1}{2}\|\sqrt{\cL_{\mu_0}} u(t)\|_{L^2}^2-\frac{1}{2}\|\sqrt{\mathcal{L}_{\mu_0}} u(t)\|_{L^2}^2\left(\frac{\|u_0\|_{L^2}}{M_{{\rm gs}} }\right)^{2(p-1)} .
$$
Therefore, if $\left\|u_0\right\|_{L^2} < M_{{\rm gs}}$, then $\|\sqrt{\cL_{\mu_0}} u(t)\|_{L^2}$ is bounded uniformly in $t > 0$. \medskip

\noindent (iii) \textbf{Case 3:} $\max\{2,\frac{d+\alpha+2}{d}\} <p< \frac{d+\alpha}{d-2}$. Multiplying \eqref{eqB1} by $\left\|u_0\right\|_{L^2}^\kappa$ with $\kappa$ given in \eqref{kappa}, 
we obtain
\begin{equation}\label{eqB2}
\begin{aligned}
    E(u_0)\|u_0\|_{L^2}^\kappa & \geq \frac{1}{2}\|\sqrt{\cL_{\mu_0}} u(t)\|_{L^2}^2\|u_0\|_{L^2}^\kappa-\frac{1}{2p \sC^{2p}}\left(\|\sqrt{\cL_{\mu_0}} u(t)\|_{L^2}^2\|u_0\|_{L^2}^\kappa\right)^{p \theta} \\
& =P\left(\|\sqrt{\cL_{\mu_0}} u(t)\|_{L^2}^2\|u_0\|_{L^2}^\kappa\right),
\end{aligned}
\end{equation}
where
$$
P(s):=\frac{1}{2} s-\frac{1}{2p \sC^{2p}} s^{p \theta}, \quad s \geq 0 .
$$
Note that $P$ has a unique critical point $s^*=\left(\frac{\sC^{2p}}{\theta}\right)^{\frac{1}{p \theta-1}}$ which corresponds to its maximum point. Combining \eqref{eq:CHGN-Q} and the Poho\v{z}aev identities \eqref{eq:Pohozaev-1} gives
$$s^*=\left(\frac{\sC^{2p}}{\theta}\right)^{\frac{1}{p \theta-1}} = H_{{\rm gs}}^2 M_{{\rm gs}}^\kappa.$$
Moreover, by using the Pohozaev identities \eqref{eq:Pohozaev-1}, the definition of $E_{{\rm gs}}$, we can show that
$$
P(s^*)=\left(\frac{1}{2}-\frac{1}{2p\theta}\right)\left( \Big(\frac{1-\theta}{\theta}\Big)^{1-\frac{\theta}{2} } \|Q\|_{L^2}^{2p-2}\right)^{ \frac{1}{p\theta -1}} =E_{{\rm gs}} M_{{\rm gs}}^\kappa.
$$

Now we prove that if \eqref{eqB3} holds then
\begin{equation}\label{eqB4}
\|\sqrt{\cL_{\mu_0}} u(t)\|_{L^2}^2\|u_0\|_{L^2}^\kappa<s^*
\end{equation}
for all $t>0$. First, (\ref{eqB4}) holds at $t=0$ by the second condition in \eqref{eqB3}. Moreover, from (\ref{eqB2}) and the first condition in \eqref{eqB3} it follows that
$$
P\left(\|\sqrt{\cL_{\mu_0}} u(t)\|_{L^2}^2\|u_0\|_{L^2}^\kappa\right) \leq E(u_0)\|u_0\|_{L^2}^\kappa < E_{{\rm gs}} M_{{\rm gs}}^\kappa=P(s^*).
$$
Therefore, since $P$ is strictly increasing in $[0, s^*]$, by the continuity of the mapping $t \mapsto\|\sqrt{\cL_{\mu_0}} u(t)\|_{L^2}^2$ $\|u_0\|_{L^2}^\kappa$, we conclude that $\|\sqrt{\cL_{\mu_0}} u(t)\|_{L^2}^2\|u_0\|_{L^2}^\kappa$ will never reach the maximum point $s^*$, namely \eqref{eqB4} holds true for all $t>0$. Consequently, \eqref{eqB4} implies that $\|\sqrt{\mathcal{L}_{\mu_0}} u(t)\|_{L^2}$ is bounded uniformly in $t>0$, which ensures the global existence of $u$.
\end{proof}

We end this section by a remark regarding the orbital stability of solutions to \eqref{eq:NLS}. 
\begin{remark} For $a>0$, define
\begin{align*} 
\cA_a&:=\{v \in \cQ_{\mu_0}: v \text{ is radial with } \| v \|_{L^2} = a \}, \\
\kappa_a&:= \inf\{E(v):  v \in \cA_a \}, \quad \quad \cS_a:=\{v \in \cA_a: E(v)=\kappa_a \}.
\end{align*}
By adapting the argument in the proof of \cite[Theorem 1.1]{TraZog-15}, it could be shown that for any $a>0$, the set $\cS_a$ is \textit{orbitally stable}, namely for any $\varepsilon>0$, there exists $\delta>0$ such that for any global solution $u$ of \eqref{eq:NLS} with ${\rm dist}(u_0,\cS_a)<\delta$, there holds 
${\rm dist}(u(t),\cS_a)<\varepsilon$ for every $t \geq 0$,
where 
$$ {\rm dist}(v,\cS_a):=\inf_{w \in \cS} \| v - w \|_{\cQ}.
$$
\end{remark}

\subsection{Finite-time blow-up} \label{subsec:blowup}

We recall the Virial identity from \cite[page 7664]{Suz-20}.

\begin{proposition} \label{prop:Virial}
	Assume $2<p<\frac{d+\alpha}{d-2}$ and $u_0 \in \Sigma_*$. Let $u$ be a solution to \eqref{eq:NLS} on $[0,T]$. Then for any $t \in [0,T]$, $u(t) \in \Sigma_*$ and the following identities hold
	\begin{align} 
		\label{exp:Virial u-dder}
		\Gamma_u''(t) &=16E(u_0) + 8(d+\alpha+2-dp)G(u(t)), \quad t \in [0,T],
	\end{align}
	where 
	\begin{equation} \label{VRI} \Gamma_u(t) = \| x u(t) \|_{L^2}^2, \quad t \in [0,T].
	\end{equation}
\end{proposition} 

\begin{proof}[\textbf{Proof of part (2) of Theorem \ref{th:global-blowup}}] We will employ the argument in \cite{Suz-17} for the case of negative energy and the argument in \cite{KenMer-08,HolRou-07} for the other case. 

Recall that $\Gamma_u(t)$ is defined in \eqref{VRI}. 
Suppose by contradiction that $u$ is a global solution of \eqref{eq:NLS} on $[0, \infty)$. First we are going to prove the following

\textbf{Claim:} If $E(u_0)<0$ or 
\begin{equation} \label{cond-Eu-Q}
E(u_0)\left\|u_0\right\|_{L^2}^\kappa<E_{\rm gs}M_{\rm gs}^\kappa, \quad\|\sqrt{\cL_{\mu_0}} u_0\|_{L^2}^2\|u_0\|_{L^2}^\kappa>H_{\rm gs}^2 M_{\rm gs}^\kappa,
\end{equation}
where $\kappa$ is given in \eqref{kappa}, then there exists a constant $a>0$ depending only on $d,\alpha,p,u_0$ such that for any $T>0$,
\begin{equation}\label{eqB5}
\Gamma^{\prime \prime}(t) \leq-a<0, \quad \forall t \in [0,T].
\end{equation}
 
Indeed, if $E\left(u_0\right)<0$, then from the Virial identity \eqref{exp:Virial u-dder} and the fact that $p>\frac{d+\alpha+2}{d}$, we derive  that
$$
\Gamma^{\prime \prime}(t)=16 E(u_0)+\frac{4(d+\alpha+2-dp)}{p} \int_{\R^d}( I_\alpha * |u(t,\cdot)|^p)|u(t, x)|^p d x \leq 16 E(u_0)<0 .
$$

Now instead of $E\left(u_0\right)<0$, we assume \eqref{cond-Eu-Q}. We adapt the argument of Holmer-Roudenko in \cite{HolRou-07}. We note that
$$
\|\sqrt{\cL_{\mu_0}} u_0\|_{L^2}^2\left\|u_0\right\|_{L^2}^\kappa>H_{\rm gs}^2 M_{\rm gs}^\kappa=s^*
$$
and by \eqref{eqB2} and the first inequality in \eqref{cond-Eu-Q}, for any $t \in (0,T)$,
\begin{equation} \label{P<Ps*}
P\left(\|\sqrt{\mathcal{L}_{\mu_0}} u(t)\|_{L^2}^2\left\|u_0\right\|_{L^2}^\kappa\right) \leq E\left(u_0\right)\left\|u_0\right\|_{L^2}^\kappa<E_{\rm gs} M_{\rm gs}^\kappa=P(s^*).
\end{equation}
Since $P$ is strictly increasing in $[0,s^*]$ and strictly decreasing in $[s^*, \infty)$, by the continuity of $t \mapsto\|\sqrt{\cL_{\mu_0}} u(t)\|_{L^2}^2\|u_0\|_{L^2}^\kappa$, we deduce from \eqref{P<Ps*} that
\begin{equation}\label{eqB6}
\|\sqrt{\cL_{\mu_0}} u(t)\|_{L^2}^2\|u_0\|_{L^2}^\kappa>s^*, \quad \forall t \in (0,T).
\end{equation}
Finally, multiplying both sides of the Virial identity\eqref{exp:Virial u-dder} by $\|u_0\|_{L^2}^\kappa$, then using \eqref{eqB6} together with the assumption $p>\frac{d+\alpha+2}{d}$ and the relation
$H_{\rm gs}^2= \frac{2(dp-d-\alpha)}{dp-d-\alpha-2}E_{\rm gs}$,
we obtain
\begin{align*}
\Gamma_{u}''(t)\|u_0\|_{L^2}^\kappa & =8(dp-d-\alpha)  E(u_0)\|u_0\|_{L^2}^\kappa - 4(dp-d-\alpha-2)\|\sqrt{\cL_{\mu_0}} u(t)\|_{L^2}^2\|u_0\|_{L^2}^\kappa \\
& \leq 8(dp-d-\alpha) \left(E(u_0)\|u_0\|_{L^2}^\kappa-E_{\rm gs}M_{\rm gs}^\kappa\right)<0 .
\end{align*}
Therefore \eqref{eqB5} holds true.

By Taylor's expansion, 
$0 \leq \Gamma_u(t)=\Gamma_u(0)+t \Gamma_u'(0)+\frac{t^2}{2} \Gamma_u''(t^*)$ for some  $t^* \in[0, t]$,
we infer from \eqref{eqB5} that
$0 \leq \Gamma_u(t) \leq \Gamma_u(0)+t \Gamma_u'(0)-\frac{t^2}{2} a$ for all $t \in [0,T]$. In particular, this implies that $aT^2 \leq \Gamma_u(0) + T\Gamma_u'(0)$ for any $T>0$, which is a contradiction. Thus $u$ must blow up in finite time.
\end{proof}
\section{Minimal mass blow-up solutions} \label{sec:minimalmass}
In this section, we will prove Theorem \ref{thm:min-mass-sol} in the mass-critical case $p=\frac{d+\alpha +2}{d}$. In this case the sharp constant in the Hardy-Gagliardo-Nirenberg inequality \eqref{eq:CHGN-def} is
\begin{equation}\label{bp1}
\sC=\left(\frac{d}{d+2+\alpha}\right)^{\frac{d}{2(d+2+\alpha)}}M_{\rm gs}^{\frac{2+\alpha}{d+2+\alpha}} .
\end{equation}

We start with the following result.

\begin{lemma}[Mass concentration as $t \rightarrow T$] \label{lem:massconcentration}
   Assume $d-2<\alpha<d$ and $p=\frac{d+\alpha+2}{d}$. Let $u$ be a solution of \eqref{eq:NLS} in $[0, T)$ with $\left\|u_0\right\|_{L^2}=M_{\rm gs}$ such that $\lim _{t \nearrow_T}\|\sqrt{\cL_{\mu_0}} u(t)\|_{L^2}=+\infty$. Let $\{t_n\} \subset (0,T)$ be an increasing sequence converging to $T$ and denote $u_n(x)=u\left(t_n, x\right)$. Then
$$
\left|u_n\right|^2 \rightarrow M_{\rm gs}^2 \delta_0 \quad \text { in }\cD'(\R^d).
$$
Here $\cD'(\R^d)$ denotes the space of distributions in $\R^d$ and $\delta_0$ is the Dirac measure concentrated at $x=0$.
\end{lemma}

\begin{proof}
For any $n \in \N$, set
$v_n(x)=\lambda_n^{\frac{d}{2}} u_n(\lambda_n x)$ for $x \in \R^d$,  with $\lambda_n=\frac{H_{\rm gs}}{\|\sqrt{\cL_{\mu_0}} u_n\|_{L^2}}$. 
Then $\lim _{n \rightarrow \infty} \lambda_n=0$ and
\begin{align} \label{vn-norm}
& \|v_n\|_{L^2}=\|u_n\|_{L^2}=\|u_0\|_{L^2}, \\ \label{Lvn-norm}
& \|\sqrt{\mathcal{L}_{\mu_0}} v_n\|_{L^2}=\lambda_n\|\sqrt{\mathcal{L}_{\mu_0}} u_n\|_{L^2}=H_{\rm gs}, \\ \label{Choquard-norm}
& \| (I_\alpha *|v_n|^p )|v_n|^p \|_{L^1} = \lambda_n^2 \|(I_\alpha *|u_n|^p )|u_n|^p \|_{L^1}.
\end{align}
Using the above identities and the energy conservation, we obtain
$$
\begin{aligned}
E(v_n) & =\frac{1}{2}\|\sqrt{\cL_{\mu_0}} v_n\|_{L^2}^2-\frac{d}{2(d+2+\alpha)}\| (I_\alpha *|v_n|^p)|v_n|^p\|_{L^1} \\
& =\lambda_n^2\left(\frac{1}{2}\|\sqrt{\cL_{\mu_0}} u_n\|_{L^2}^2-\frac{d}{2(d+2+\alpha)}  \|(I_\alpha *|u_n|^p )|u_n|^p \|_{L^1} \right) \\ 
& = \lambda_n^2 E(u_n)=\lambda_n^2 E(u_0) \to 0.
\end{aligned}
$$
It follows that
$$
\lim _{n \to \infty}\int_{\R^d}(I_\alpha *|v_n|^p )|v_n|^p dx =\lim _{n \to \infty}\left(\frac{d+2+\alpha}{d}\right)\|\sqrt{\cL_{\mu_0}} v_n\|^2_{L^2}=\left(\frac{d+2+\alpha}{d}\right)H_{\rm gs}^2,
$$
and hence
$$
\lim _{n \to \infty} \frac{\|\sqrt{\cL_{\mu_0}} v_n\|_{L^2}^{\frac{d}{d+2+\alpha}}\left\|v_n\right\|_{L^2}^{\frac{2+\alpha}{d+2+\alpha}}}{  \|(I_\alpha *|v_n|^p)|v_n|^p\|_{L^1}^\frac{1}{2p}}=\left(\frac{d}{d+2+\alpha}\right)^{\frac{d}{2(d+2+\alpha)}}M_{\rm gs}^{\frac{2+\alpha}{d+2+\alpha}}=\sC.
$$
It means that $\left\{v_n\right\}$ is a minimizing sequence for \eqref{eq:CHGN-def}.
By Theorem \eqref{precompact}, there exist a subsequence, still denoted by $\left\{v_n\right\}$, a ground state $Q \in \cG$ and constants $\lambda>0$, $z \in \mathbb{C}$ such that $v_n(x) \rightarrow z Q(\lambda x)$ in $\cQ_{\mu_0}$. We deduce from \eqref{vn-norm} and \eqref{Lvn-norm} that $
\lambda=|z|=1$. 
In particular, we obtain
$|v_n|^2 \to|Q|^2$  in  $L^1(\R^d)$.

It follows that
\begin{equation}\label{bp2}
\begin{aligned}
&\left|\int_{\R^d}\left|u_n(y)\right|^2 \phi(y) dy-\|Q\|_{L^2}^2 \phi(0)\right| \\ &\leq \|\phi\|_{L^{\infty}} \int_{\R^d}|| v_n(x)|^2-|Q(x)|^2|  d x 
+\int_{\R^d}|Q(x)|^2\left|\phi(\lambda_n x)-\phi(0)\right| d x .
\end{aligned}
\end{equation}
Since $|v_n|^2 \to Q^2$ in $L^1(\R^d)$, the first term on the right-hand side of \eqref{bp2} tends to zero as $n \to \infty$. Moreover, since $\lambda_n \to 0$ and $\phi \in C_c^{\infty}(\R^d)$ it follows that $\phi(\lambda_n x) \to\phi(0)$ as $n \to \infty$ for all $x \in \R^d$. By invoking the dominated convergence theorem, we derive that the second term on the right-hand side of \eqref{bp2}) tends to zero as $n \to \infty$. As a consequence, $|u_n|^2 \to M_{\rm gs}^2 \delta_0$ in $\cD'(\R^d)$.
\end{proof}

\begin{lemma} \label{lem:nabla-phase} 
Assume  $v \in \cQ_{\mu_0}$,  $\phi \in C_0^\infty(\R^d,\R)$ radial and $s \in \R$.

(i) Then we have
\begin{equation}\label{eq:uei-energy}
	E(ve^{is\phi})=E(v)+s \, \int_{\R^d}   \nabla \phi \cdot \Im  (\overline{v} \nabla v) dx+\frac{s^2}{2}\int_{\R^d} |\nabla \phi|^2 |v|^2 dx.
\end{equation}

(ii) Assume in addition that $p=\frac{d+\alpha+2}{d}$ and $\|v\|_{L^2}=M_{\rm gs}$. Then we have
\begin{equation} \label{est:im}
	\left|\int_{\R^d} \nabla \phi \cdot \Im(\bar{v} \nabla v) d x\right| \leqslant \sqrt{2 E(v)}\left(\int_{\R^d}|\nabla \phi|^2|v|^2 d x\right)^{\frac{1}{2}}.
\end{equation}
\end{lemma}
\begin{proof}
(i) By using \cite[identity (4.22)]{CsoGen-18}, the fact that $C_0^\infty(\R^d \setminus \{0\})$ is dense in $\cQ_{\mu_0}$ and the and the standard density argument,  we obtain identity \eqref{eq:uei-energy}. The detailed proof is left to the reader.

(ii) Since $\| ve^{is\phi}\|_{L^2} = \| v\|_{L^2} = M_{\rm gs}$, by using \eqref{eq:CHGN-def} with $p=\frac{d+\alpha+2}{d}$ and \eqref{bp1}, we derive that $E(ve^{is\phi}) \geq 0$, which implies that
$$E(v)+s \, \int_{\R^d}   \nabla \phi \cdot \Im  (\overline{v} \nabla v) dx+\frac{s^2}{2}\int_{\R^d} |\nabla \phi|^2 |v|^2 dx \geq 0, \quad \forall s \in \R.$$
This implies \eqref{est:im}.
\end{proof}

\begin{lemma}
Let $d-2<\alpha<d$ and $p=\frac{d+\alpha+2}{d}$. Assume $u$ is a solution of \eqref{eq:NLS} in $[0, T)$ such that $\|u_0\|_{L^2}=M_{\rm gs}$ and $\lim _{t \nearrow_T}\| \sqrt{\cL_{\mu_0}}u(t)\|_{L^2} = +\infty$. Then for all $t \in[0, T)$ we have $|x| u(t) \in L^2(\R^d)$ and
\begin{equation} \label{Virial-masscritical}
\Gamma_u(t) = 8 E\left(u_0\right)(T-t)^2.
\end{equation}
In particular, 
\begin{align} \label{Virial-Gamma(0)}
	\Gamma_u(0) & =\int_{\R^d}|x|^2|u_0|^2 d x=8 E(u_0) T \\ \label{Virial-Gamma'(0)}
	\Gamma_u'(0) & =4 \Im \int_{\R^d} \overline{x u_0(x)} \cdot \nabla u_0(x) d x=-16 E(u_0) T .
\end{align}
\end{lemma}
\begin{proof} By employing the same argument as in the proof of \cite[Lemma 15]{MukNamNgu-2021} and using Lemma \ref{lem:nabla-phase}, we obtain the desired result. The detail proof is omitted. 
\end{proof}

 For any $T>0$, $\lambda>0$, $\gamma \in \R$ and $Q \in \cG$, define
\begin{equation} \label{S}
	\bS_{T,\lambda,\gamma}^{Q}(t,x): = e^{i \gamma}e^{i\frac{\lambda^2}{T-t}}e^{-i\frac{|x|^2}{4(T-t)}}\left( \frac{\lambda}{T-t} \right)^{\frac{d}{2}} Q\left( \frac{\lambda x}{T-t} \right) \quad x\in \R^d, t \in [0,T).
\end{equation}

\begin{proof}[\textbf{Proof of Theorem \ref{thm:min-mass-sol}}]
	
It can be checked that for any $T>0$, $\lambda>0$, $\gamma \in \R$ and $Q \in \cG$, the function $u=\bS_{T,\lambda,\gamma}^Q$ is a minimal-mass solution of \eqref{eq:NLS} which blows up in finite time $T>0$. Therefore, we conclude statement 1.
	
Next we are going to prove statement 2 by adapting the strategy of Hmidi-Keraani \cite{HmiKer-05}. 
By applying \eqref{eq:uei-energy} with $v=u_0$, $s=1/(2T)$ and  $\phi(x)$ approaching $|x|^2/2$ (using appropriate cut-off functions) in  and using the identities \eqref{Virial-Gamma(0)} and \eqref{Virial-Gamma'(0)}, we obtain 
\begin{equation*}
	\begin{aligned}
		E(u_0e^{\frac{i|x|^2}{4T}})&=E(u_0)+\frac{1}{2T} \Im \int_{\R^d} \overline{x u_0} \cdot \nabla u_0 dx+\frac{1}{8T^2}\int_{\R^d} |x|^2 |u_0|^2dx\\
		&=E(u_0)-\frac{4E(u_0)T}{2T}+\frac{1}{8T^2}(8E(u_0)T^2)=0.
	\end{aligned}
\end{equation*}
This implies that the function 
$v_0 = u_0 e^{\frac{i|x|^2}{4T}}$
satisfies that $\|v_0\|_{L^2} = M_{\rm gs}$ and $E(v_0)=0$. Hence $v_0$ is a minimizer for the Hardy-Gagliardo-Nirenberg inequality \eqref{eq:CHGN-def}. Then there exist $\lambda_1>0, \sigma \in \R$ and $Q \in \cG$ such that 
$$u_0(x) e^{\frac{i|x|^2}{4T}}=e^{i \sigma} \lambda_1^{\frac{d}{2}}Q(\lambda_1 x), \quad  x \in \R^d.$$
This yields
$$
u_0(x)=e^{i\sigma}e^{-\frac{i|x|^2}{4T}}\lambda_1^{\frac{d}{2}}Q(\lambda_1 x) \quad x \in \R^d.
$$
Setting $\lambda_0=\lambda_1 T>0$,  $\gamma_0=\sigma-\lambda_1^2T$, we arrive at
$$ u_0(x)=e^{i \gamma_0}e^{i \frac{\lambda_0^2}{T}}e^{-i\frac{|x|^2}{4T}} \left( \frac{ \lambda_0}{T}\right)^{\frac{d}{2}}Q\left(\frac{\lambda_0 x}{T}\right)=\bS_{T,\lambda_0,\gamma_0}^Q(0,x) \quad x \in \R^d.
$$ 
By the uniqueness, we conclude that
$u(t,x)=\bS_{T,\lambda_0,\gamma_0}^Q(t,x)$ for all $t \in [0,T)$. 
\end{proof}

\appendix
\section{Some auxiliary results}
\begin{lemma}\label{lem:vanishing}
Let $\left\{f_n\right\}$ be a bounded sequence in $\cQ_{\mu_0}$ and assume that $\{\left|f_n\right|^2\}$ vanishes in the sense that
$$
\limsup _{n \rightarrow \infty}\left(\sup _{y \in \mathbb{R}^d} \int_{|x-y| \leqslant R}\left|f_n(x)\right|^2 d x\right)=0 \quad \text { for any } R>0.
$$
Then for any $2 < q < 2^*$, $\|f_n\|_{L^q} \to 0$ as $n \to \infty$.
\end{lemma}
\begin{proof}
By the continuous embedding \eqref{eq:Q-in-Lp}, we deduce that $\{f_n\}$ is a bounded sequence in $H^{s}(\R^d)$ for any $s \in (0,1)$ and in $L^r(\R^d)$ for any $r \in [2,2^*)$. In particular,  $\{f_n\}$ is a bounded sequence in $H^{\frac{1}{2}}(\R^d)$. By \cite[Lemma 7.2]{LenLew-10}, $\| f_n \|_{L^q} \to 0$ for any $q \in (2,\frac{2d}{d-1})$ (note that \cite[Lemma 7.2]{LenLew-10} is proved for $d=3$ but by the same argument can be employed to obtain the result for $d \geq 3$). By interpolation and the fact that $\{f_n\}$ is bounded in  $L^r(\R^d)$ for any  $2 \leq r < 2^*$, we derive that $\| f_n \|_{L^q} \to 0$ as $n \to \infty$  for any  $2 < r < 2^*$ .
\end{proof}

\begin{lemma} \label{lem:extract}
Let $\{u_n\}$ be a sequence in $\cQ_{\mu_0}$ such that $u_n \wto u$ weakly in $\cQ_{\mu_0}$ and let $0 \leq R_k \leq R_k'$ such that $R_k \to \infty$. Then there exists a subsequence $\{u_{n_k}\}$ such that
\begin{equation} \label{extract-1}
\int_{B_{R_k}}|u_{n_k}|^2 dx \to \int_{\R^d}|u|^2 dx \quad \text{and} \quad \int_{A_{R_k,R_k'}}|u_{n_k}|^2 dx \to 0,
\end{equation}
as $k \to \infty$.  In particular, $1_{B_{R_{n_k}}(0)} \to u$ strongly in $L^q(\R^d)$ for any $2 \leq q < 2^*$.
\end{lemma}
\begin{proof}
By using the argument in the proof of \cite[Lemma 14]{Lew-10}, together with the embedding \eqref{eq:Q-in-Lp}, we obtain the desired result.
\end{proof}

\bibliographystyle{siam}

\end{document}